\documentclass[leqno]{siamltex}
\usepackage{amsmath}
\usepackage{graphicx}
\usepackage{mathrsfs}
\usepackage{bm}
\usepackage{float}
\usepackage{amsfonts,amssymb}
\usepackage{stmaryrd} 
\usepackage{dsfont}
\usepackage{pifont}
\usepackage{tikz}
\usepackage{wrapfig} 
\usepackage{hyperref}
\usepackage{multirow}
\usepackage{lineno}
\usepackage{mathtools}
\usepackage{appendix}
\usepackage{color}
\usepackage{caption}
\usepackage{subfigure}

\newtheorem{remark}{Remark}
\newtheorem{example}{Example}

\newcommand{\pT}{{\partial \mathcal{T}}}
\def\3bar{{|\!|\!|}}
\newcommand{\T}{{\mathcal{T}}}

\newcommand{\vertiii}[1]{{\left\vert\kern-0.25ex\left\vert\kern-0.25ex\left\vert #1
    \right\vert\kern-0.25ex\right\vert\kern-0.25ex\right\vert}}

\title{A Generalized Weak Galerkin method for Oseen equation}

\author{
Wenya Qi
\thanks{School of Mathematics and Information Sciences, Henan Normal University, Xinxiang, Henan 453007, China(qiwymath@163.com). The research of W.Q. was partially supported by Natural Science Foundation of Henan Province(Grant No.222300420213).}
 \and Padmanabhan Seshaiyer
 \thanks{Department of Mathematical Sciences, George Mason University, Fairfax, VA 22030, USA(pseshaiy@gmu.edu).}
 \and Junping Wang
\thanks{Division of Mathematical Sciences, National Science Foundation, Alexandria, VA 22314, USA(jwang@nsf.gov). The research of J.W. was supported by the NSF IR/D program, while working at National Science Foundation. However, any opinion, finding, and conclusions or recommendations expressed in this material are those of the author and do not necessarily reflect the views of the National Science Foundation.}}

\begin{document}

\maketitle
\begin{abstract}
In this work, the authors introduce a generalized weak Galerkin (gWG) finite element method for the time-dependent Oseen equation. The generalized weak Galerkin method is based on a new framework for approximating the gradient operator. Both a semi-discrete and a fully-discrete numerical scheme are developed and analyzed for their convergence, stability, and error estimates.
A generalized {\em{inf-sup}} condition is developed to assist the convergence analysis. The backward Euler discretization is employed in the design of the fully-discrete scheme. Error estimates of optimal order are established mathematically, and they are validated numerically with some benchmark examples.
\end{abstract}



\pagestyle{myheadings}

\section{Introduction}\label{sec:1}
Navier-Stokes equation has been used to describe incompressible viscous flow and forms the backbone of fluid mechanics \cite{temam2001}. One of the difficulties in the equation is the nonlinear convective acceleration term which involves the product of the velocity with its gradient. One approximation to the Navier-Stokes equation is the Stokes equation where this convective term is omitted. Oseen equation, in comparison to Stokes flow, includes a partial inclusion of the convective acceleration and a reactive term in the momentum equations \cite{Chester1966, oseen1927}.
Oseen equation is also considered as an auxiliary problem to linearize Navier-Stokes equation \cite{posta2005, Huang2014, Zhang2015, Shang2018}. 

The research on the development, analysis and application of the finite element method for the Oseen equation continues to receive attention in the computational sciences community.
In \cite{John2016},  the framework of finite element method for Oseen equation was established and some stabilized schemes were analyzed. 
A general a posteriori error estimator was designed for problems with incompressibility constraint and shown to be valid for the Oseen equation in \cite{Ainsworth1997}. 
The local discontinuous Galerkin method using mixed setting has been introduced for the steady Oseen equation in \cite{Cockburn2003},
and the optimal convergence was derived.
A stabilized finite element method was derived for Oseen equation in \cite{Barrenechea2002, Barrenechea2004} through the introduction of a stabilization parameter and mesh dependent term.
In \cite{Burman2006}, with the estimates independent of local Reynolds number, the interior penalty finite element method was discussed for steady-state Oseen equation.
The space-time discontinuous Galerkin discretization was introduced and analyzed for the evolutionary Oseen equation with time-dependent flow domain in \cite{Vegt2008}.
By using subgrid scale methods, a stabilized finite element formulation was presented for Oseen equation in \cite{Codina2008}, which also proved the stability and optimal convergence.
For optimal control of the Oseen equation, SUPG/PSPG were analyzed in \cite{braack2012}, and a priori analysis was provided. In view of stabilization tensor,  the hybridizable discontinuous Galerkin was provided for steady Oseen equation in \cite{Cesm2013}.
A pressure stabilized finite element scheme was estimated for  the evolutionary Oseen equation in \cite{Notsu2015}.
Least-squares finite element method was presented in \cite{cai2016} where the convergence in various norms was proved. 
Based on a residual type a posteriori error estimator and pressure projection method, a stabilized finite volume method was analyzed in \cite{lu2018}.
In \cite{khan2020}, a divergence-conforming discontinuous method for steady Oseen equation was established 
and a fully robust posteriori error estimator was derived.
Based on the skew-symmetry scheme and two grid discretizations, local and parallel finite element algorithms were established for the evolutionary Oseen equation in  \cite{ding2021},
and error estimates for semi-discrete and fully-discrete were derived.
A multiscale hybrid-mixed method was presented for Oseen equation in \cite{Araya2021}, and a posteriori error estimator was analyzed for the adaptive method.

Weak Galerkin method was first introduced for elliptic problem in \cite{Wang2013weak, Wang2014a}, 
and weak function was presented for the definition of weak gradient operator.
Weak Galerkin method has been used for Stokes equation in \cite{wang2016stokes} with the use of weak gradient and weak divergence.
A weak Galerkin method with skew-symmetry scheme was introduced in \cite{liu2016} for steady-state Oseen equation by using a weak convective operator. 
By introducing weak trilinear term, a skew-symmetry weak Galerkin method was analyzed for Navier-Stokes equation in \cite{liu2018},  
and linearized Oseen approximation was employed in numerical experiments.
For extending the finite element piecewise polynomial spaces to arbitrary finite dimensional spaces,
the generalized weak Galerkin method {(gWG)} has been introduced with new generalized weak gradient definition.
 In particular, gWG  has been established for steady Stokes equation in \cite{qi2021} 
 with the finite element spaces of arbitrary combination of polynomials.
Based on the same spaces as  \cite{qi2021},  we consider the steady and evolutionary Oseen equation by using gWG scheme in this work.
The main contribution of this work includes the development of a scheme that is proposed from the equation directly without the skew-symmetry terms, 
and the coercivity of bilinear form which is derived for the uniqueness and existence of solutions. 

This work is organized as follows. In Section \ref{sec:2}, generalized weak gradient and  the related weak divergence are presented. 
The semi-discrete and fully-discrete {gWG} schemes  for the evolutionary Oseen equation are established with backward Euler for time discretization.
In Section \ref{sec:3}, the {gWG} scheme is introduced for steady-state Oseen equation firstly, and optimal convergence orders are obtained.
Then, for evolutionary case, the error estimate is presented by introducing an elliptic projection operator in Section \ref{sec:4}.
Finally, numerical examples for benchmark problems are shown to verify the theoretical results in Section \ref{sec:5}.

In this article, we adopt the notations in \cite{adams2003sobolev} for Sobolev spaces and denote generic nonnegative constants by $C$ which is independent of the mesh size and time step.
 
\section{gWG for Evolutionary Oseen Equation}\label{sec:2}
 We consider the evolutionary Oseen equation that seeks the unknown fluid velocity $\mathbf{u}(t): =\mathbf{u}(\cdot, t)$ and the pressure $p(t): =p(\cdot, t)$ such that
\begin{equation}\label{evo_oseen_problem}
\begin{aligned}
\rho\mathbf{u}_t-\mu\Delta \mathbf{u}+\rho(\beta\cdot\nabla)\mathbf{u}+\nabla p&=\mathbf{f},~~~~~~~~~~\mbox{in}~\Omega\times(0,\bar{T}], \\
\nabla\cdot\mathbf{u}&=0,~~~~~~~~~~\mbox{in}~\Omega\times(0,\bar{T}],\\
\mathbf{u}&=\mathbf{g}, ~~~~~~~~~~\mbox{on}~\partial\Omega\times(0,\bar{T}], \\
\mathbf{u}(0)&=\mathbf{g}_2, ~~~~~~~~~\mbox{in}~\Omega, 
\end{aligned}
\end{equation}
where $\Omega$ is an open bounded polygonal or polyhedral domain in $\mathbb{R}^d,~d=2,3$. 
The external force field $\mathbf{f}$, the admitted flux $\mathbf{g}$  across $\partial\Omega$ and the initial velocity $\mathbf{g}_2$ are given functions. 
$\beta$ is a given convection field, $\mu>0$ is the constant kinematic viscosity and $\rho\geq 0$ is a given scalar function. When $\rho= 0$, \eqref{evo_oseen_problem} is steady-state Stokes equation which has been analyzed with gWG in \cite{qi2021}, and we consider $\rho> 0$ in this paper.

The weak form of problem \eqref{evo_oseen_problem} reads as follows: find $\mathbf{u}(t)\in[H^1(\Omega)]^d$ and $p(t)\in L_0^2(\Omega)$ 
satisfying $\mathbf{u}(t)=\mathbf{g}$ on $\partial\Omega$  and $\mathbf{u}(0)=\mathbf{g}_2$ in $\Omega$ such that
\begin{equation}\label{evo_weakform}
\begin{aligned}
\rho(\mathbf{u}_t, \mathbf{v})+\mu(\nabla \mathbf{u}, \nabla \mathbf{v})
+{\rho}((\beta\cdot\nabla)\mathbf{u}, \mathbf{v})-(p,\nabla\cdot\mathbf{v})&=(\mathbf{f},\mathbf{v}), ~~\forall \mathbf{v}\in[H^1_{0}(\Omega)]^d,\\
(\nabla\cdot\mathbf{u}, q)&=0,~~~~~~~\forall q\in L_0^2(\Omega).
\end{aligned}
\end{equation}

Here, we denote the spaces
\begin{equation*}
\begin{aligned}
[H^1_{0}(\Omega)]^d&=\{\mathbf{v}\in [H^1(\Omega)]^d, ~\mathbf{v}=\mathbf{0} ~\mbox{on} ~\partial \Omega\},\\
L_0^2(\Omega)&=\Big\{q\in L^2(\Omega), \int_{\Omega}q dx = 0\Big\}.
\end{aligned}
\end{equation*}

Denote the partition using polygon or polyhedra of domain $\Omega$ by $\mathcal{T}_h$ which satisfy the shape regular conditions in \cite{Wang2014a}. 
For each $T\in\mathcal{T}_h$, $h_T$ is its diameter and $h=\max_{T\in\mathcal{T}_h}h_T$ is mesh size of the partition $\mathcal{T}_h$. 
Let $\partial T$ be the edges of element $T$ and $\mathcal{E}_{I}$ the set of interior edges.
Denote by ${P}_{k}(T)$ the polynomials space of degree less than or equal to $k$ in all variables on element $T$. 
We consider the finite element spaces
\begin{equation*}
\begin{aligned}
 \mathbf{V}_{h}:=&\{(\mathbf{v}_0, \mathbf{v}_b): \mathbf{v}_0|_{T}\in[{P}_{k}(T)]^d, \mathbf{v}_b|_{\partial T}\in[{P}_{j}(\partial T)]^d,~ \forall T \in \mathcal{T}_{h}\},\\
W_{h}:=&\{q\in L_0^2(\Omega): q|_{T}\in{P}_{n}(T), ~\forall T \in \mathcal{T}_{h}\},
 \end{aligned}
\end{equation*}
where $k, j, n\geq 0$ are arbitrary non-negative integers. 
Denote the subspace of $\mathbf{V}_{h}$ with vanishing boundary condition by $\mathbf{V}_{h}^0$, i.e.
$\mathbf{V}_{h}^0=\{\mathbf{v}\in \mathbf{V}_{h}, ~\mathbf{v}_b=\mathbf{0} ~\mbox{on} ~\partial \Omega\}.$

We define the local $L^2$ projection operators $Q_0: [L^2(T)]^d \rightarrow [P_k(T)]^d$ and $Q_b: [L^2(\partial T)]^d \rightarrow [P_j (\partial T)]^d$. 
Denote by $Q_h =\{Q_0, Q_b\}$ the projection operator to the space of weak functions and by $Q_h^p$ the $L^2$ projection operator onto  $W_h$.

Next, based on the weak function and projection operators, we define a weak gradient.

\begin{definition}\label{weakgradient}
For each $\mathbf{v}\in \mathbf{V}_h$, a generalized discrete weak gradient operator $\nabla_w\mathbf{v}$ is defined as follows
$$\nabla_w\mathbf{v}=\nabla\mathbf{v}_0+\delta_w\mathbf{v},$$
where $\delta_w\mathbf{v}|_T\in [P_l(T)]^{d\times d}$ is given by
$$(\delta_w\mathbf{v}, \phi)_T=\langle \mathbf{v}_b- Q_b\mathbf{v}_0,  \phi\cdot\mathbf{n}  \rangle_{\partial T}, \forall \phi\in [P_l(T)]^{d\times d},$$
where $\mathbf{n}$ denotes the unit outward normal to $\partial T$.
\end{definition}

In order to present the numerical scheme for Oseen equation, we also need to define a weak divergence operator.

\begin{definition}\label{weakdivergence}
For each $\mathbf{v}\in \mathbf{V}_h$, we define the weak divergence $\nabla_w\cdot\mathbf{v}\in P_m(T)$ such that
$$(\nabla_w\cdot\mathbf{v}, \psi)_T = - (\mathbf{v}_0, \nabla \psi)_T+\langle \mathbf{v}_b, \psi\cdot\mathbf{n} \rangle_{\partial T}, ~\forall~\psi\in P_m(T).$$
 \end{definition}

Based on the weak formulation \eqref{evo_weakform} and the definition of weak differential operators, the semi-discrete {g}WG scheme for the evolutionary Oseen equation \eqref{evo_oseen_problem} reads as follows: For $t\in(0, \bar{T}]$, find $\mathbf{u}_h(t)\in \mathbf{V}_{h}$ and $p_h(t)\in W_{h}$ 
satisfying $\mathbf{u}_b(t)=Q_b\mathbf{g}$ on $\partial \Omega$ and $\mathbf{u}_h(0)=Q_h\mathbf{g}_2$ in $\Omega$  such that
\begin{equation}\label{sch:evo_iWG_semi}
\begin{aligned}
&\rho(\mathbf{u}_{ht}, \mathbf{v}_h)+\mu(\nabla_w \mathbf{u}_h, \nabla_w \mathbf{v}_h)+{\rho}((\beta\cdot\nabla_w)\mathbf{u}_h, \mathbf{v}_0)\\
&-(p_h,\nabla_w\cdot\mathbf{v}_h)+s_1(\mathbf{u}_h, \mathbf{v}_h)=(\mathbf{f},\mathbf{v}_0), ~~~~~~~~~~~~~~~\forall~ \mathbf{v}_h\in\mathbf{V}_{h}^0,\\
&(\nabla_w\cdot\mathbf{u}_h, q_h)+s_2(p_h, q_h)=0,~~~~~~~~~~~~~~~~~~~~~~~~~~\forall~q_h\in W_{h}.
\end{aligned}
\end{equation}
Here,  the stabilization terms are defined as follows
\begin{equation*}
\begin{aligned}
s_1(\mathbf{u}_h, \mathbf{v}_h)&=\zeta\sum_{T\in\mathcal{T}_h}h_T^{\gamma}\langle Q_b(\mathbf{u}_b-\mathbf{u}_0), Q_b(\mathbf{v}_b-\mathbf{v}_0)\rangle_{\partial T},\\
s_2(p_h, q_h)&=\sigma\sum_{e\in\mathcal{E}_{I}}h_e^{\alpha}\langle \llbracket p_h \rrbracket, \llbracket q_h \rrbracket \rangle_e,
\end{aligned}
\end{equation*}
where $\llbracket q_h \rrbracket={q_h}|_{\partial T_1\cap e}-{q_h}|_{\partial T_2\cap e}$ denotes the jump on interior edge $e$ shared by the elements $T_1$ and $T_2$.
Assume that $\vert\gamma\vert<\infty$, $\vert\alpha\vert<\infty$. Here $\zeta$ is a positive parameter, and $\sigma$ is zero or one. 

For simplicity, we consider uniform time stepping $\tau=\bar{T}/N$ with $N$ being the total time steps. Denote by $t^{\bar{n}}=\bar{n}*\tau$ the $\bar{n}$-th time level.

By replacing the time derivative by backward difference quotient which is backward Euler discretization, the fully-discrete {g}WG scheme seeks $\mathbf{u}^{\bar{n}+1}\in \mathbf{V}_{h}$ and $p^{\bar{n}+1}\in W_{h}$ 
satisfying the boundary condition $\mathbf{u}_b^{\bar{n}+1}=Q_b\mathbf{g}$ on $\partial \Omega$  and initial condition $\mathbf{u}^0=Q_h\mathbf{g}_2$ in $\Omega$ such that
\begin{equation}\label{sch:evo_iWG_fully}
\begin{aligned}
&\rho(\bar{\partial}_t{\mathbf{u}^{\bar{n}+1}}, \mathbf{v}_h)+\mu(\nabla_w \mathbf{u}^{\bar{n}+1}, \nabla_w \mathbf{v}_h)+{\rho}((\beta\cdot\nabla_w)\mathbf{u}^{\bar{n}+1}, \mathbf{v}_0)\\
&-(p^{\bar{n}+1},\nabla_w\cdot\mathbf{v}_h)+s_1(\mathbf{u}^{\bar{n}+1}, \mathbf{v}_h)=(\mathbf{f}^{\bar{n}+1},\mathbf{v}_0), ~~~~~~~~~~~~~~~\forall~ \mathbf{v}_h\in\mathbf{V}_{h}^0,\\
&(\nabla_w\cdot\mathbf{u}^{\bar{n}+1}, q_h)+s_2(p^{\bar{n}+1}, q_h)=0,~~~~~~~~~~~~~~~~~~~~~~~~~~~~~~~\forall~q_h\in W_{h}.
\end{aligned}
\end{equation}
Here, the backward difference quotient is denoted by $\bar{\partial}_t{\mathbf{u}^{\bar{n}+1}}=\frac{\mathbf{u}^{\bar{n}+1}-\mathbf{u}^{\bar{n}}}{\tau}$.

Before coming to the estimates of gWG schemes \eqref{sch:evo_iWG_fully} and  \eqref{sch:evo_iWG_semi}  for evolutionary case, 
let us analyze the {g}WG for steady-state Oseen equation.
 
\section{{g}WG for Steady-State Oseen Equation}\label{sec:3}
 The steady-state Oseen equation seeks an unknown velocity $\mathbf{u}$ and a pressure $p$ such that
\begin{equation}\label{st_oseen_problem}
\begin{aligned}
-\mu\Delta \mathbf{u}+\rho(\beta\cdot\nabla)\mathbf{u}+\nabla p&=\mathbf{f},~~~~~~~~~~\mbox{in}~\Omega, \\
\nabla\cdot\mathbf{u}&=0,~~~~~~~~~~\mbox{in}~\Omega,\\
\mathbf{u}&=\mathbf{g}, ~~~~~~~~~~\mbox{on}~\partial\Omega.
\end{aligned}
\end{equation} 
By a weak solution we mean a pair of  $\mathbf{u}\in[H^1(\Omega)]^d$ and $p\in L_0^2(\Omega)$ 
satisfying $\mathbf{u}=\mathbf{g}$ on $\partial\Omega$ and the following equations
\begin{equation}\label{weakform}
\begin{aligned}
\mu(\nabla \mathbf{u}, \nabla \mathbf{v})+{\rho}((\beta\cdot\nabla)\mathbf{u}, \mathbf{v})-(p,\nabla\cdot\mathbf{v})&=(\mathbf{f},\mathbf{v}), ~~\forall \mathbf{v}\in[H^1_{0}(\Omega)]^d,\\
(\nabla\cdot\mathbf{u}, q)&=0,~~~~~~~\forall q\in L_0^2(\Omega).
\end{aligned}
\end{equation}
The {g}WG scheme for the steady-state Oseen equation \eqref{st_oseen_problem} seeks $\mathbf{u}_h\in \mathbf{V}_{h}$ and $p_h\in W_{h}$ satisfying $\mathbf{u}_b=Q_b\mathbf{g}$ on $\partial \Omega$ and the following equations 
\begin{equation}\label{sch:WG_FEM}
\begin{aligned}
&\mu(\nabla_w \mathbf{u}_h, \nabla_w \mathbf{v}_h)+{\rho}((\beta\cdot\nabla_w)\mathbf{u}_h, \mathbf{v}_0)
-(p_h,\nabla_w\cdot\mathbf{v}_h)+s_1(\mathbf{u}_h, \mathbf{v}_h)=(\mathbf{f},\mathbf{v}_0),\\
&(\nabla_w\cdot\mathbf{u}_h, q_h)+s_2(p_h, q_h)=0,
\end{aligned}
\end{equation}
for all $\mathbf{v}_h\in\mathbf{V}_{h}^0$ and $q_h\in W_{h}$.

For convenience, we introduce the following energy semi-norm in $\mathbf{V}_h$:
$$
\3bar \mathbf{v}_h \3bar^2:=(\nabla_w \mathbf{v}_h, \nabla_w \mathbf{v}_h)+s_1(\mathbf{v}_h, \mathbf{v}_h), \quad v_h\in \mathbf{V}_h.
$$
\begin{lemma}\label{lem:norm} $\3bar \cdot\3bar$ defines a norm in the subspace $\mathbf{V}_h^0$ consisting of functions with vanishing boundary value.
\end{lemma}
\begin{proof}
For $\mathbf{v}_h\in \mathbf{V}_h^0$ with $\3bar \mathbf{v}_h\3bar=0$,
we see that $\nabla_w\mathbf{v}_h=\nabla\mathbf{v}_0+\delta_w\mathbf{v}_h=0$ on each element $T$ and $Q_b\mathbf{v}_0=\mathbf{v}_b$ on each $\partial T$.
By using the definition of weak gradient, we have 
\begin{equation*}
\begin{aligned}
(\delta_w\mathbf{v}_h, \delta_w\mathbf{v}_h)_T=\langle\mathbf{v}_b- Q_b\mathbf{v}_0, \delta_w\mathbf{v}_h \cdot\mathbf{n}\rangle_{\partial T}=0.
\end{aligned}
\end{equation*}
It follows that $\delta_w\mathbf{v}_h=0$ and hence, $\nabla\mathbf{v}_0=0$, which implies that $\mathbf{v}_0$ is a constant on each element. Thus, we have $Q_b\mathbf{v}_0=\mathbf{v}_b$ on each $\partial T$.
Together with $\mathbf{v}_b=0$ on $\partial\Omega$, we obtain $\mathbf{v}_h\equiv 0$.
\end{proof}

In addition, for simplicity, we denote $b(\mathbf{v}_h, q_h):=(\nabla_w\cdot\mathbf{v}_h, q_h)$.
\begin{lemma}\label{lem:inf-sup}\cite{qi2021}
Assume that $n\leq \min \{m, k+1\}$.
For each $q_h\in W_h$, there exists $\mathbf{v}_h\in \mathbf{V}_h^0$ such that
\begin{equation}\label{inf_supconditon}
\begin{aligned}
b(\mathbf{v}_h, q_h) &\geq \frac{1}{2}\Vert q_h\Vert^2-Ch^{1-\alpha}\sum_{e\in\mathcal{E}_{I}}h_e^{\alpha}\Vert  \llbracket q_h \rrbracket \Vert_e^2,\\
\3bar \mathbf{v}_h\3bar&\leq C(1+h^{\frac{1+\gamma}{2}}) \Vert q_h\Vert.
\end{aligned}
\end{equation}
\end{lemma}

Next, we set $s=\min\{j, l\}$ and introduce a local $L^2$ projection operator $\bm{Q}_s: [L^2(T)]^{d\times d} \rightarrow [P_s(T)]^{d\times d}$. From the construction of the generalized weak gradient operator,  we have the following results.
 \begin{lemma}\label{lem:identity}\cite{qi2021}
For any $\mathbf{v}\in \mathbf{V}_h$ and $\mathbf{w}\in [H^1(T)]^d$ on the element $T\in\mathcal{T}_h$, 
one has for each $\phi \in [P_s(T)]^{d\times d}$
 \begin{equation*}
\begin{aligned}
(\nabla_w\mathbf{v}, \phi)_T&=-(\mathbf{v}_0, \nabla\cdot\phi)_T+\langle \mathbf{v}_b, \phi\cdot\mathbf{n}\rangle_{\partial T},\\
(\nabla_wQ_h\mathbf{w}, \phi)_T&=(\nabla\mathbf{w}, \phi)_T+((I-Q_0)\mathbf{w}, \nabla\cdot\phi)_T.
 \end{aligned}
\end{equation*}
\end{lemma}

The following is a useful result for the analysis on convergence and solution existence of the gWG scheme.
\begin{lemma}\label{lem:coercive} Assume $\gamma\leq 0$ and $-\frac{\nabla\cdot\beta}{2}\geq C_0\geq0$.
There exists a constant $\zeta=\frac{C\Vert\beta\Vert_{\infty}+\Vert\beta\Vert_{\infty}^2}{2\rho}$ such that
\begin{equation*}
\begin{aligned}
((\beta\cdot\nabla_w)\mathbf{v}_h, \mathbf{v}_0) 
\geq C_0\Vert\mathbf{v}_0\Vert^2-\frac{1}{2\rho}s_1(\mathbf{v}_h, \mathbf{v}_h)-\epsilon\Vert\nabla_w\mathbf{v}_h\Vert^2,
\end{aligned}
\end{equation*}
for all $\mathbf{v}_h\in \mathbf{V}_h^0$.
\end{lemma}
\begin{proof}From the definition of the weak gradient, we have for each $\mathbf{w}_h\in V_h^0$
\begin{equation*}
\begin{aligned}
&((\beta\cdot\nabla_w)\mathbf{v}_h, \mathbf{w}_0)_T =(\nabla\mathbf{v}_0+\delta_w\mathbf{v}_h, \beta\cdot\mathbf{w}_0^t)_T \\
=&-(\mathbf{v}_0,\nabla\cdot( \beta\cdot\mathbf{w}_0^t))_T +\langle\mathbf{v}_0,  \beta\cdot\mathbf{w}_0^t\cdot\mathbf{n}\rangle_{\pT}+(\delta_w\mathbf{v}_h, \beta\cdot\mathbf{w}_0^t)_T \\
=&-(\mathbf{v}_0,\nabla\cdot\beta\mathbf{w}_0)_T-(\beta\cdot\mathbf{v}_0^t,\nabla\mathbf{w}_0)_T +\langle\mathbf{v}_0,  \beta\cdot\mathbf{w}_0^t\cdot\mathbf{n}\rangle_\pT+(\delta_w\mathbf{v}_h, \beta\cdot\mathbf{w}_0^t)_T \\
=&-(\mathbf{v}_0,\nabla\cdot\beta\mathbf{w}_0)_T-(\beta\cdot\mathbf{v}_0^t,\nabla_w\mathbf{w}_h)_T +(\beta\cdot\mathbf{v}_0^t, \delta_w\mathbf{w}_h)_T\\
&+\langle\mathbf{v}_0,  \beta\cdot\mathbf{w}_0^t\cdot\mathbf{n}\rangle_\pT+(\delta_w\mathbf{v}_h, \beta\cdot\mathbf{w}_0^t)_T. 
\end{aligned}
\end{equation*}
From Definition \ref{weakgradient} and $\langle\mathbf{v}_0, \beta\cdot\mathbf{v}_b^t\cdot\mathbf{n}\rangle_\pT=\langle\mathbf{v}_b, \beta\cdot\mathbf{v}_0^t\cdot\mathbf{n}\rangle_\pT$, we have by taking $\mathbf{w}_h=\mathbf{v}_h$ in the above identity,
\begin{equation*}
\begin{aligned}
((\beta\cdot\nabla_w)\mathbf{v}_h, \mathbf{v}_0)_T=&-\frac{1}{2}(\mathbf{v}_0,\nabla\cdot\beta\mathbf{v}_0)_T+
\langle\mathbf{Q}_s(\beta\cdot\mathbf{v}_0^t)\cdot\mathbf{n}, \mathbf{v}_b-Q_b\mathbf{v}_0\rangle_\pT\\
&+\frac{1}{2}\langle\mathbf{v}_0, \beta\cdot\mathbf{v}_0^t\cdot\mathbf{n}\rangle_\pT
+((I-\mathbf{Q}_s)\beta\cdot\mathbf{v}_0^t, \delta_w\mathbf{v}_h)_T\\
=&-\frac{1}{2}(\mathbf{v}_0,\nabla\cdot\beta\mathbf{v}_0)_T-\langle\mathbf{Q}_s(\beta\cdot\mathbf{v}_0^t)\cdot\mathbf{n}, \mathbf{v}_0-\mathbf{v}_b\rangle_\pT \\
&+\frac{1}{2}\langle\mathbf{v}_0, \beta\cdot\mathbf{v}_0^t\cdot\mathbf{n}\rangle_\pT
+((I-\mathbf{Q}_s)\beta\cdot\mathbf{v}_0^t, \delta_w\mathbf{v}_h)_T\\
=&-\frac{1}{2}(\mathbf{v}_0,\nabla\cdot\beta\mathbf{v}_0)_T+\langle(I-\mathbf{Q}_s)(\beta\cdot\mathbf{v}_0^t)\cdot\mathbf{n}, \mathbf{v}_0-\mathbf{v}_b\rangle_\pT \\
&+ ((I-\mathbf{Q}_s)\beta\cdot\mathbf{v}_0^t, \delta_w\mathbf{v}_h)_T
-\frac{1}{2}\langle\mathbf{v}_0-\mathbf{v}_b, \beta\cdot(\mathbf{v}_0-\mathbf{v}_b)^t\cdot\mathbf{n}\rangle_\pT\\
&+\frac12 \langle\mathbf{v}_b, \beta\cdot\mathbf{v}_b^t\cdot\mathbf{n}\rangle_\pT.
\end{aligned}
\end{equation*}
By summing over all elements $T\in \T_h$ and noticing 
$\sum_{T\in\T_h}\langle\mathbf{v}_b, \beta\cdot\mathbf{v}_b^t\cdot\mathbf{n}\rangle_\pT=0$ we obtain
\begin{equation*}
\begin{aligned}
((\beta\cdot\nabla_w)\mathbf{v}_h, \mathbf{v}_0)_{\T_h}
=&-\frac{1}{2}(\mathbf{v}_0,\nabla\cdot\beta\mathbf{v}_0)_{\T_h}+\langle(I-\mathbf{Q}_s)(\beta\cdot\mathbf{v}_0^t)\cdot\mathbf{n}, \mathbf{v}_0-\mathbf{v}_b\rangle_{\pT_h} \\
&+ ((I-\mathbf{Q}_s)\beta\cdot\mathbf{v}_0^t, \delta_w\mathbf{v}_h)_{\T_h}
-\frac{1}{2}\langle\mathbf{v}_0-\mathbf{v}_b, \beta\cdot(\mathbf{v}_0-\mathbf{v}_b)^t\cdot\mathbf{n}\rangle_{\pT_h}.
\end{aligned}
\end{equation*}
Observe the following estimate from the property of $L^2$ projections:
\begin{equation}\label{eq:nabla1}
\begin{aligned}
\Vert\nabla\mathbf{v}_0\Vert_T\leq&\Vert\nabla_w\mathbf{v}_h\Vert_T+\Vert\delta_w\mathbf{v}_h\Vert_T
\leq\Vert\nabla_w\mathbf{v}_h\Vert_T+Ch^{-\frac{1}{2}}\Vert\mathbf{v}_b-Q_b\mathbf{v}_0\Vert_{\partial T}
\end{aligned}
\end{equation}
and 
\begin{equation}\label{eq:nabla2}
\begin{aligned}
\Vert\mathbf{v}_b-\mathbf{v}_0\Vert_{\partial T}\leq&\Vert\mathbf{v}_b-Q_b\mathbf{v}_0\Vert_{\partial T}+\Vert Q_b\mathbf{v}_0-\mathbf{v}_0\Vert_{\partial T}
\leq\Vert\mathbf{v}_b-Q_b\mathbf{v}_0\Vert_{\partial T}+Ch^{\frac{1}{2}}\Vert \nabla\mathbf{v}_0\Vert_{T}\\
\leq&C\Vert\mathbf{v}_b-Q_b\mathbf{v}_0\Vert_{\partial T}+Ch^{\frac{1}{2}}\Vert\nabla_w\mathbf{v}_h\Vert_{T}.
\end{aligned}
\end{equation}
From the Cauchy-Schwarz inequality, \eqref{eq:nabla1}, \eqref{eq:nabla2}, the construction of $\delta_w\mathbf{v}_h$,  and the trace inequality, we obtain
\begin{equation*}
\begin{aligned}
&\left|\langle(I-\mathbf{Q}_s)(\beta\cdot\mathbf{v}_0^t)\cdot\mathbf{n}, \mathbf{v}_0-\mathbf{v}_b\rangle_{\pT_h}+((I-\mathbf{Q}_s)\beta\cdot\mathbf{v}_0^t, \delta_w\mathbf{v}_h)_{\T_h}\right|\\
\leq &C \Vert\beta\Vert_{\infty}\sum_{T\in\mathcal{T}_h}h^{\frac{1}{2}}\Vert\nabla\mathbf{v}_0\Vert_T\Big( \Vert \mathbf{v}_0-\mathbf{v}_b\Vert_{\partial T}+\Vert Q_b\mathbf{v}_0-\mathbf{v}_b\Vert_{\partial T}\Big)\\
\leq& Ch\Vert\nabla\mathbf{v}_0\Vert^2 +  Ch\Vert\nabla_w\mathbf{v}_h\Vert^2
+\Vert\beta\Vert^2_{\infty}\Vert Q_b\mathbf{v}_0-\mathbf{v}_b\Vert^2_{\partial T}\\
\leq& Ch\Vert\nabla_w\mathbf{v}_h\Vert^2+\frac{\Vert\beta\Vert^2_{\infty}}{\zeta}h^{-\gamma}s_1(\mathbf{v}_h, \mathbf{v}_h).
\end{aligned}
\end{equation*}

Next, using the Cauchy-Schwarz inequality and \eqref{eq:nabla2} again leads to  
\begin{equation*}
\begin{aligned}
\left|\langle\mathbf{v}_0-\mathbf{v}_b, \beta\cdot(\mathbf{v}_0-\mathbf{v}_b)^t\cdot\mathbf{n}\rangle_{\pT_h}\right|
\leq&  \sum_{T\in\mathcal{T}_h}\Vert\beta\Vert_{\infty}\Vert\mathbf{v}_b-\mathbf{v}_0\Vert_{\partial T}^2\\
\leq&Ch\Vert\nabla_w\mathbf{v}_h\Vert^2+\frac{C\Vert\beta\Vert_{\infty}}{\zeta}h^{-\gamma}s_1(\mathbf{v}_h, \mathbf{v}_h).
\end{aligned}
\end{equation*}
By combining the above inequality with $\gamma\leq 0$, we get
\begin{equation*}
\begin{aligned}
&((\beta\cdot\nabla_w)\mathbf{v}_h, \mathbf{v}_0) 
\geq C_0\Vert\mathbf{v}_0\Vert^2-Ch\Vert\nabla_w\mathbf{v}_h\Vert^2
-(C\Vert\beta\Vert_{\infty}+\Vert\beta\Vert_{\infty}^2)\frac{1}{\zeta}s_1(\mathbf{v}_h, \mathbf{v}_h).
\end{aligned}
\end{equation*}
By choosing $(C\Vert\beta\Vert_{\infty}+\Vert\beta\Vert_{\infty}^2)\frac{1}{\zeta}=\frac{1}{2\rho}$, 
i.e $\zeta=\frac{C\Vert\beta\Vert_{\infty}+\Vert\beta\Vert_{\infty}^2}{2\rho}$, we arrive at
\begin{equation*}
\begin{aligned}
&((\beta\cdot\nabla_w)\mathbf{v}_h, \mathbf{v}_0) \geq C_0\Vert\mathbf{v}_0\Vert^2-Ch\Vert\nabla_w\mathbf{v}_h\Vert^2
-\frac{1}{2\rho}s_1(\mathbf{v}_h, \mathbf{v}_h).
\end{aligned}
\end{equation*}
When $h\rightarrow 0$, we have $\epsilon=Ch\rightarrow 0$. This completes the proof of the lemma.
\end{proof}
\medskip

Using Lemmas \ref{lem:norm}, \ref{lem:inf-sup}  and \ref{lem:coercive}, we obtain the following theorem.
\begin{theorem}\label{th:nique}
Assume $\nabla\cdot\beta\leq 0$. The generlized weak Galerkin method \eqref{sch:WG_FEM} has one and only one solution for sufficiently small meshsize $h$.
\end{theorem}

From the assumption, there exists $C_0$ such that $-\frac{\nabla\cdot\beta}{2}\geq C_0\geq0$. It follows that 
\begin{equation}\label{eq:coercive}
\begin{aligned}
&\mu(\nabla_w \mathbf{v}_h, \nabla_w \mathbf{v}_h)+\rho((\beta\cdot\nabla_w)\mathbf{v}_h, \mathbf{v}_0)+s_1(\mathbf{v}_h, \mathbf{v}_h)\\ 
\geq& (\mu-\epsilon)\Vert\nabla_w\mathbf{v}_h\Vert^2+C_0\rho\Vert\mathbf{v}_0\Vert^2+\frac{1}{2}s_1(\mathbf{v}_h, \mathbf{v}_h)\\
\geq& C_1\Vert\nabla_w\mathbf{v}_h\Vert^2+C_0\rho\Vert\mathbf{v}_0\Vert^2+\frac{1}{2}s_1(\mathbf{v}_h, \mathbf{v}_h),
\end{aligned}
\end{equation}
where $C_1>0$ for sufficiently small meshsize $h$. The coercivity estimate \eqref{eq:coercive} is a key component behind the proof of Theorem \ref{th:nique}. The proof of Theorem \ref{th:nique} is routine and details are left to interested readers.

\subsection{Error Equations of gWG for Steady-State Oseen Equation}\label{sec:3.1}
In this subsection, we shall derive the error equations for the velocity and pressure approximations. 
Denote the velocity error by $e_h^u={Q}_h\mathbf{u}-\mathbf{u}_h$ and the pressure error by $e_h^p={Q}_h^p p-p_h$.

From the properties of projection and Lemma \ref{lem:identity}, we obtain the following identity \cite{qi2021}, 
\begin{equation}\label{eq:spid2}
\begin{aligned}
(\nabla_wQ_h\mathbf{u}, \nabla_w \mathbf{v}_h)
=&-(\Delta\mathbf{u}, \mathbf{v}_0)- \langle (I-\bm{Q}_s)\nabla\mathbf{u}\cdot\mathbf{n}, \mathbf{v}_b-\mathbf{v}_0\rangle\\
&+((\bm{Q}_s-I)\nabla\mathbf{u}, \nabla\mathbf{v}_0)+((I-Q_0)\mathbf{u}, \nabla \cdot \bm{Q}_s\nabla_w \mathbf{v}_h)\\
&+(\nabla_wQ_h\mathbf{u}, (I-\bm{Q}_s)\nabla_w \mathbf{v}_h).
\end{aligned}
\end{equation}
For each $\mathbf{v}_h\in \mathbf{V}_h$, it follows from the definition of the generalized weak gradient and Lemma \ref{lem:identity} that
\begin{equation}\label{eq:first_term}
\begin{aligned}
&((\beta\cdot\nabla_w)Q_h\mathbf{u}, \mathbf{v}_0)=(\nabla_wQ_h\mathbf{u}, \beta\mathbf{v}_0^t)\\
=&(\nabla_wQ_h\mathbf{u}, \mathbf{Q}_s(\beta\mathbf{v}_0^t))+(\nabla_wQ_h\mathbf{u}, (I-\mathbf{Q}_s)\beta\mathbf{v}_0^t)\\
=&(\nabla\mathbf{u}, \mathbf{Q}_s(\beta\mathbf{v}_0^t))+((I-Q_0)\mathbf{u}, \nabla\cdot \mathbf{Q}_s(\beta\mathbf{v}_0^t))+(\nabla_wQ_h\mathbf{u}, (I-\mathbf{Q}_s)\beta\mathbf{v}_0^t)\\
=&(\nabla\mathbf{u}, \beta\mathbf{v}_0^t)-(\nabla\mathbf{u}, (I-\mathbf{Q}_s)\beta\mathbf{v}_0^t)+((I-Q_0)\mathbf{u}, \nabla\cdot \mathbf{Q}_s(\beta\mathbf{v}_0^t))\\
&+(\nabla_wQ_h\mathbf{u}, (I-\mathbf{Q}_s)\beta\mathbf{v}_0^t).
\end{aligned}
\end{equation}
With  $k-1\leq n\leq m$, we note that \cite{qi2021}
\begin{equation}\label{eq:sec_ter}
\begin{aligned}
(Q_h^p p, \nabla_w \cdot\mathbf{v}_h)&= -(\nabla p, \mathbf{v}_0)-\sum_{T\in\mathcal{T}_h} \langle (I-Q_h^p) p, (\mathbf{v}_b- \mathbf{v}_0)\cdot \mathbf{n}\rangle_{\partial T},
\end{aligned}
\end{equation}
Combining \eqref{eq:spid2} -\eqref{eq:sec_ter} leads to 
\begin{equation}\label{eq:project_term_st}
\begin{aligned}
&\mu(\nabla_wQ_h\mathbf{u}, \nabla_w\mathbf{v}_h)+\rho((\beta\cdot\nabla_w)Q_h\mathbf{u},\mathbf{v}_0)
-(Q_h^p p, \nabla_w\cdot\mathbf{v}_h)+s_1(Q_h\mathbf{u}, \mathbf{v}_h)\\
=&(\mathbf{f}, \mathbf{v}_0)+L^1_{u,p}(\mathbf{v}_h),
\end{aligned}
\end{equation}
where
\begin{equation*}
\begin{aligned}
L^1_{u,p}(\mathbf{v}_h):=&\mu\Big(\langle(I-\mathbf{Q}_s)\nabla\mathbf{u}\cdot\mathbf{n}, \mathbf{v}_0-\mathbf{v}_b\rangle-((I-\mathbf{Q}_s)\nabla\mathbf{u},\nabla\mathbf{v} _0)\\
&+((I-Q_0)\mathbf{u}, \nabla\cdot \mathbf{Q}_s\nabla_w\mathbf{v}_h)+(\nabla_wQ_h\mathbf{u}, (I-\mathbf{Q}_s)\nabla_w\mathbf{v}_h) \Big)\\
&+\rho\Big(-(\nabla\mathbf{u}, (I-\mathbf{Q}_s)\beta\mathbf{v}_0^t)+((I-Q_0)\mathbf{u}, \nabla\cdot \mathbf{Q}_s(\beta\mathbf{v}_0^t))\\
&+(\nabla_wQ_h\mathbf{u}, (I-\mathbf{Q}_s)\beta\mathbf{v}_0^t) \Big)
-\langle(I-Q_h^p)p, (\mathbf{v}_0-\mathbf{v}_b)\cdot\mathbf{n}\rangle+s_1(Q_h\mathbf{u},\mathbf{v}_h).
\end{aligned}
\end{equation*}
From \eqref{sch:WG_FEM} and \eqref{eq:project_term_st}, we arrive at the following error equations 
\begin{equation}\label{eq:error_eq_st}
\begin{aligned}
&\mu(\nabla_we^u_h, \nabla_w\mathbf{v}_h)+((\beta\cdot\nabla_w)e^u_h,\mathbf{v}_0)
-(e_h^p, \nabla_w\cdot\mathbf{v}_h)+s_1(e^u_{h}, \mathbf{v}_h)=L^1_{u,p}(\mathbf{v}_h),\\
&(\nabla_w\cdot e_h^u, q_h)+s_2(e_h^p, q_h)=L^2_{u,p}(q_h)
\end{aligned}
\end{equation}
for all $\mathbf{v}_h\in\mathbf{V}_h^0$ and $q_h\in W_h$, 
where
$$L^2_{u,p}(q_h)=-\langle(I-Q_b)\mathbf{u}, q_h\cdot\mathbf{n}\rangle+s_2(Q_h^p p, q_h).$$

\subsection{Convergence of gWG for Steady-State Oseen Equation}\label{sec:3.2}
In this subsection, our goal is to establish the error estimates in $L^2$ norm and energy norm for the gWG approximations for steady-state Oseen equation. 

\begin{lemma} Assume $\mathbf{u}\in [H^{i}(\Omega)]^d$ where $i=\max\{s+2,k+1,j+1\}$ and $p\in H^{n+1}(\Omega)$.  There exists a constant $C$ such that 
\begin{equation}\label{eq:error_term_es_st}
\begin{aligned}
\left|L^1_{u,p}(\mathbf{v}_h)\right|\leq& C\Big((1+h^{\frac{-\gamma-1}{2}})(h^{s+1}\Vert\mathbf{u}\Vert_{s+2}+h^{k}\Vert\mathbf{u}\Vert_{k+1} 
+h^{n+1}\Vert p\Vert_{n+1})\\
&+h^{\frac{\gamma+1}{2}}h^{k}\Vert\mathbf{u}\Vert_{k+1}\Big)\3bar\mathbf{v}_h\3bar,\\
\left|L^2_{u,p}(q_h)\right|\leq &C\Big(h^j h^{\frac{1-\alpha}{2}}\Vert\mathbf{u}\Vert_{j+1}+h^{n+1}h^{\frac{\alpha-1}{2}}\Vert p\Vert_{n+1}\Big)s_2^{\frac{1}{2}}(q_h,q_h)
\end{aligned}
\end{equation}
for all $\mathbf{v}_h\in\mathbf{V}_h^0$ and $q_h\in W_h$,
\end{lemma}
\begin{proof}
From the Cauchy-Schwarz inequality, trace inequality, and approximation properties of $L^2$ projections, we have
\begin{equation*}
\begin{aligned}
|\langle(I-\mathbf{Q}_s)\nabla\mathbf{u}\cdot\mathbf{n}, \mathbf{v}_0-\mathbf{v}_b\rangle|
&\leq C\sum_{T\in\mathcal{T}_h}\Vert(I-\mathbf{Q}_s)\nabla\mathbf{u}\Vert_{\partial T}\Vert\mathbf{v}_0-\mathbf{v}_b\Vert_{\partial T}\\
&\leq Ch^{s+1}(1+h^{\frac{-\gamma-1}{2}})\Vert\mathbf{u}\Vert_{s+2}\3bar\mathbf{v}_h\3bar,
\end{aligned}
\end{equation*}
where we have used $\Vert\mathbf{v}_0-\mathbf{v}_b\Vert_{\partial T}\leq Ch^{\frac{1}{2}}(1+h^{\frac{-\gamma-1}{2}})\3bar\mathbf{v}_h\3bar$.
Analogously, we have 
\begin{equation*}
\begin{aligned}
|((I-\mathbf{Q}_s)\nabla\mathbf{u},\nabla\mathbf{v} _0)|\leq C\Vert(I-\mathbf{Q}_s)\nabla\mathbf{u} \Vert\Vert\nabla\mathbf{v} _0\Vert
\leq Ch^{s+1}(1+h^{\frac{-\gamma-1}{2}})\Vert\mathbf{u} \Vert_{s+2}\3bar\mathbf{v} _h\3bar,
\end{aligned}
\end{equation*}
where we have used the following inequalities
\begin{equation*}
\begin{aligned}
\Vert\delta_w\mathbf{v} _h\Vert^2&=\langle\mathbf{v} _b-Q_b\mathbf{v} _0, \delta_w\mathbf{v} _h\rangle
\leq C\sum_{T\in\mathcal{T}_h}h^{\frac{\gamma}{2}}\Vert\mathbf{v} _b-Q_b\mathbf{v} _0\Vert_{\partial T}h^{\frac{-\gamma-1}{2}}\Vert\delta_w\mathbf{v} _h\Vert_{T}
\end{aligned}
\end{equation*}
and 
$\Vert\nabla\mathbf{v} _0\Vert\leq \Vert\nabla_w\mathbf{v} _h\Vert+\Vert\delta_w\mathbf{v} _h\Vert\leq C(1+\frac{-\gamma-1}{2})\3bar\mathbf{v} _h\3bar$.
 
Next, from the inverse inequality we have 
\begin{equation*}
\begin{aligned}
|((I-Q_0)\mathbf{u}, \nabla\cdot \mathbf{Q}_s\nabla_w\mathbf{v}_h)|\leq Ch^k\Vert\mathbf{u} \Vert_{k+1}\3bar\mathbf{v} _h\3bar.
\end{aligned}
\end{equation*}
By using the definition of generalized weak gradient, we obtain
\begin{equation}\label{eq:L1esti}
\begin{aligned}
&|(\nabla_wQ_h\mathbf{u}, (I-\mathbf{Q}_s)\nabla_w\mathbf{v}_h)|=|(\nabla Q_0\mathbf{u}+\delta_wQ_h\mathbf{u}, (I-\mathbf{Q}_s)\nabla_w\mathbf{v}_h)|\\
\leq& C\Big(\Vert(I-\mathbf{Q}_s)\nabla Q_0\mathbf{u}\Vert+\Vert\delta_wQ_h\mathbf{u}\Vert\Big)\Vert(I-\mathbf{Q}_s)\nabla_w\mathbf{v}_h\Vert\\
\leq& C\Big(h^{s+1}\Vert\mathbf{u}\Vert_{s+2}+h^{-1}\Vert(I-Q_0)\mathbf{u}\Vert\Big)\3bar\mathbf{v} _h\3bar\\
\leq& C\Big(h^{s+1}\Vert\mathbf{u}\Vert_{s+2}+h^{k}\Vert\mathbf{u}\Vert_{k+1}\Big)\3bar\mathbf{v} _h\3bar,
\end{aligned}
\end{equation}
where we have used the fact that $\Vert\delta_wQ_h\mathbf{u}\Vert\leq Ch^{-1}\Vert(I-Q_0)\mathbf{u}\Vert$.

Next, the following estimates hold true: 
\begin{equation*}
\begin{aligned}
|(\nabla\mathbf{u}, (I-\mathbf{Q}_s)\beta\cdot\mathbf{v}_0^t)|&=|((I-\mathbf{Q}_s)\nabla\mathbf{u}, (I-\mathbf{Q}_s)\beta\cdot\mathbf{v}_0^t)|\\
&\leq C h^{s+1}\Vert\mathbf{u}\Vert_{s+2} h\Vert \nabla\mathbf{v}_0\Vert\\
&\leq Ch^{s+1}\Vert\mathbf{u}\Vert_{s+2}h(1+h^{\frac{-\gamma-1}{2}})\3bar\mathbf{v} _h\3bar.
\end{aligned}
\end{equation*}
Similar to the above inequality and \eqref{eq:L1esti}, we have 
\begin{equation*}
\begin{aligned}
|((I-Q_0)\mathbf{u}, \nabla\cdot \mathbf{Q}_s\beta\mathbf{v}_0^t)|
&\leq Ch^{k+1}\Vert\mathbf{u}\Vert_{k+1}(1+h^{\frac{-\gamma-1}{2}})\3bar\mathbf{v} _h\3bar,\\
|(\nabla_wQ_h\mathbf{u}, (I-\mathbf{Q}_s)\beta\mathbf{v}_0^t)|&
\leq C(h^{s+1}\Vert\mathbf{u}\Vert_{s+2}+h^{k}\Vert\mathbf{u}\Vert_{k+1})h(1+h^{\frac{-\gamma-1}{2}})\3bar\mathbf{v} _h\3bar.
\end{aligned}
\end{equation*}
From the construction of the stabilizer, we get
\begin{equation*}
\begin{aligned}
&|\langle(I-Q_h^p)p, (\mathbf{v}_0-\mathbf{v}_b)\cdot\mathbf{n}\rangle|\leq C(1+h^{\frac{-\gamma-1}{2}})h^{n+1}\Vert p\Vert_{n+1}\3bar\mathbf{v} _h\3bar,\\
&|s_1(Q_h\mathbf{u}, \mathbf{v}_h)|=\left|\sum_{T\in\mathcal{T}_h}\zeta h^{\gamma}\langle Q_b(I-Q_0)\mathbf{u}, \mathbf{v}_b-Q_b\mathbf{v}_0\rangle_{\partial T}\right|
\leq Ch^{\frac{\gamma+1}{2}}h^k\Vert\mathbf{u}\Vert_{k+1}\3bar\mathbf{v} _h\3bar.
\end{aligned}
\end{equation*} 
Based on the above inequalities, we obtain the desired estimate for $L^1_{u,p}(\mathbf{v}_h)$. 

Finally, from the following estimates
\begin{equation*}
\begin{aligned}
|\langle(I-Q_b)\mathbf{u}, q_h\cdot\mathbf{n}\rangle|&\leq Ch^j\Vert\mathbf{u}\Vert_{j+1}h^{\frac{-\alpha+1}{2}}s_2(q_h,q_h)^{\frac{1}{2}},\\
|s_2(Q_h^p p, q_h)|&\leq Ch^{n+1}\Vert p\Vert_{n+1}h^{\frac{\alpha-1}{2}}s_2(q_h,q_h)^{\frac{1}{2}},
\end{aligned}
\end{equation*} 
we may deduce the desired estimate for $L^2_{u, p}(q_h)$ as well. This completes the proof of the lemma.
\end{proof}

\begin{theorem}\label{th:ste_energy_er}
Let $(\mathbf{u}, p)$ be  the exact solution of \eqref{weakform} and $(\mathbf{u}_h, p_h)$ be its numerical approximation arising from  the {\em{g}}WG scheme \eqref{sch:WG_FEM}.
Assume that the velocity $\mathbf{u}\in [H^{i}(\Omega)]^d$, $i=\max\{s+2,k+1,j+1\}$ and the pressure $p\in H^{n+1}(\Omega)$. Let ${k-1}\leq n\leq \max\{m, k+1\}$ be satisfied. Then we have
\begin{equation}\label{erres_steady_energy}
\begin{aligned}
\3bar e_h^{u}\3bar^2&\leq CR_1,\\
\Vert e_h^{p}\Vert^2 &\leq C(1+h^{\gamma+1}+h^{2(1-\alpha)})R_1,
\end{aligned}
\end{equation}
where 
\begin{equation*}
\begin{aligned}
R_1=&h^{2(s+1)}(1+h^{-\gamma-1})\Vert \mathbf{u}\Vert^2_{s+2}+h^{2k}(1+h^{-\gamma-1}+h^{\gamma+1})\Vert \mathbf{u}\Vert^2_{k+1}\\
&+h^{2j}h^{1-\alpha}\Vert \mathbf{u}\Vert^2_{j+1}
+h^{2(n+1)}(1+h^{-\gamma-1}+h^{\alpha-1})\Vert p\Vert^2_{n+1}.
\end{aligned}
\end{equation*}
Moreover, when $\gamma=-1$ and $\alpha=1$, optimal order of convergence is achieved.
\end{theorem}
\begin{proof}
By choosing $\mathbf{v}_h=e_h^{u}$ and $q_h=e_h^p$ in \eqref{eq:error_eq_st},  we have
\begin{equation*}
\begin{aligned}
\mu\Vert\nabla_we_h^u\Vert^2+\rho((\beta\cdot\nabla_w)e_h^{u},e_h^{u})+s_1(e_h^u,e_h^u)+s_2(e^p_h,e^p_h)=L^1_{u,p}(e_h^u)+L^2_{u,p}(e_h^p).
\end{aligned}
\end{equation*}
From Lemma  \ref{lem:coercive} or the coercivity \eqref{eq:coercive} and  the estimate in \eqref{eq:error_term_es_st}, we have 
\begin{equation*}
\begin{aligned}
&C_1\Vert\nabla_we_h^u\Vert^2+ C_0\rho\Vert e_{h0}^u\Vert^2+\frac{1}{2}s_1(e_h^u,e_h^u)+s_2(e^p_h,e^p_h) \leq CR_1,
\end{aligned}
\end{equation*}
which gives rise to the error estimate for the velocity approximation in the energy norm.

Next, by using Lemma \ref{lem:inf-sup} and \eqref{eq:error_eq_st}, for $e_h^p$, there exists a $\mathbf{v}_h\in \mathbf{V}_h^0$ such that  
\begin{equation*}
\begin{aligned}
&\Vert e_h^p\Vert^2\leq b(\mathbf{v}_h, e_h^p)+Ch^{1-\alpha}s_2(e^p_h,e^p_h)\\
\leq&\mu(\nabla_we^u_h, \nabla_w\mathbf{v}_h)+\rho((\beta\cdot\nabla_w)e^u_h,\mathbf{v}_0)
+s_1(e^u_{h}, \mathbf{v}_h)-L^1_{u,p}(\mathbf{v}_h)+Ch^{1-\alpha}s_2(e^p_h,e^p_h).
\end{aligned}
\end{equation*}
Thus, from the estimate \eqref{eq:error_term_es_st}, we obtain the desired estimate for $\Vert e_h^p\Vert$ in the $L^2$ norm.
\end{proof}

In order to estimate the velocity in $L^2$ norm, we consider the dual problem for the steady-state Oseen equation \eqref{st_oseen_problem}: Seek $(\phi,\xi)$ such that 
\begin{equation}\label{dual_st_problem}
\begin{aligned}
-\mu\Delta \phi-\rho(\beta\cdot\nabla)\phi-\rho\nabla\cdot\beta\phi+\nabla \xi&=e_{h0}^u,~~~~~~~~~~\mbox{in}~\Omega, \\
\nabla\cdot\phi&=0,~~~~~~~~~~~~~\mbox{in}~\Omega,\\
\phi&=0, ~~~~~~~~~~~~~\mbox{on}~\partial\Omega.
\end{aligned}
\end{equation} 
Assume that the dual problem satisfies the following $H^2/H^1$ regularity assumption: 
$$\Vert\phi\Vert_2+\Vert\xi\Vert_1\leq C\Vert e_{h0}^u\Vert.$$

\begin{theorem}\label{th:ste_L2_er} Let $(\mathbf{u}, p)$ be  the exact solution of \eqref{weakform} and $(\mathbf{u}_h, p_h)$ be its numerical approximation arising from the {\em{g}}WG scheme \eqref{sch:WG_FEM}.  
Assume that the velocity $\mathbf{u}\in [H^{i}(\Omega)]^d$, $i=\max\{s+2,k+1,j+1\}$ and the pressure $p\in H^{n+1}(\Omega)$. Let ${k-1}\leq n\leq \max\{m, k+1\}$ and $j\geq 1$ be satisfied. 
Then there exists a constant such that
\begin{equation}\label{erres_steady_l2}
\begin{aligned}
&\Vert e_{h0}^{u}\Vert ^2\leq C(h^2+h^{2\alpha})(1+h^{\gamma+1}+h^{2(1-\alpha)})R_1.
\end{aligned}
\end{equation}
\end{theorem}
Moreover, when $\gamma=-1$ and $\alpha=1$, the above error estimate yields the optimal order of convergence in $L^2$.
\begin{proof}
First, notice that 
\begin{equation*}
\begin{aligned}
&(Q_0\phi, (\beta\cdot\nabla_w)\mathbf{v}_h)
=(\phi, (\beta\cdot\nabla_w)\mathbf{v}_h)-((I-Q_0)\phi, (\beta\cdot\nabla_w)\mathbf{v}_h).
\end{aligned}
\end{equation*}
Next, from the definition of the generalized weak gradient we have
\begin{equation*}
\begin{aligned}
&(\phi, (\beta\cdot\nabla_w)\mathbf{v}_h)=(\beta\phi^t, \nabla_w\mathbf{v}_h)\\
=&(\mathbf{Q}_s\beta\phi^t, \nabla_w\mathbf{v}_h)+((I-\mathbf{Q}_s)\beta\phi^t, \nabla_w\mathbf{v}_h)\\
=&(\mathbf{Q}_s\beta\phi^t, \nabla\mathbf{v}_0)-\langle \mathbf{Q}_s\beta\phi^t\cdot\mathbf{n}, \mathbf{v}_0-\mathbf{v}_b\rangle+((I-\mathbf{Q}_s)\beta\phi^t, \nabla_w\mathbf{v}_h)\\
=&(\beta\phi^t, \nabla\mathbf{v}_0)-((I-\mathbf{Q}_s)\beta\phi^t, \nabla\mathbf{v}_0)-\langle \mathbf{Q}_s\beta\phi^t\cdot\mathbf{n}, \mathbf{v}_0-\mathbf{v}_b\rangle\\
&+((I-\mathbf{Q}_s)\beta\phi^t, \nabla_w\mathbf{v}_h).
\end{aligned}
\end{equation*}
It follows that
\begin{equation}\label{eq:second_term}
\begin{aligned}
&(Q_0\phi, (\beta\cdot\nabla_w)\mathbf{v}_h)
=(\beta\phi^t, \nabla\mathbf{v}_0)-((I-\mathbf{Q}_s)\beta\phi^t, \nabla\mathbf{v}_0)-\langle \mathbf{Q}_s\beta\phi^t\cdot\mathbf{n}, \mathbf{v}_0-\mathbf{v}_b\rangle\\
&+((I-\mathbf{Q}_s)\beta\phi^t, \nabla_w\mathbf{v}_h)-((I-Q_0)\phi, (\beta\cdot\nabla_w)\mathbf{v}_h)\\
=&-(\nabla\cdot(\beta\phi^t), \mathbf{v}_0)+\langle\beta\phi^t, \mathbf{v}_0\cdot\mathbf{n}\rangle-((I-\mathbf{Q}_s)\beta\phi^t, \nabla\mathbf{v}_0)-\langle \mathbf{Q}_s\beta\phi^t\cdot\mathbf{n}, \mathbf{v}_0-\mathbf{v}_b\rangle\\
&+((I-\mathbf{Q}_s)\beta\phi^t, \nabla_w\mathbf{v}_h)-((I-Q_0)\phi, (\beta\cdot\nabla_w)\mathbf{v}_h).
\end{aligned}
\end{equation}
From the property of $L^2$ projections, \eqref{eq:spid2}, \eqref{eq:sec_ter}, and \eqref{eq:second_term}, we have 
\begin{equation}\label{eq:project_term_st_dual}
\begin{aligned}
&\mu(\nabla_wQ_h\phi, \nabla_w\mathbf{v}_h)+\rho(Q_0\phi, (\beta\cdot\nabla_w)\mathbf{v}_h)
-(Q_h^p \xi, \nabla_w\cdot\mathbf{v}_h)+s_1(Q_h\phi, \mathbf{v}_h)\\
=&(e_{h0}^u, \mathbf{v}_0)+\rho\langle\beta\phi^t, \mathbf{v}_0\cdot\mathbf{n}\rangle+\mu\Big(\langle(I-\mathbf{Q}_s)\nabla\phi\cdot\mathbf{n}, \mathbf{v}_0-\mathbf{v}_b\rangle-((I-\mathbf{Q}_s)\nabla\phi,\nabla\mathbf{v} _0)\\
&+((I-Q_0)\phi, \nabla\cdot \mathbf{Q}_s\nabla_w\mathbf{v}_h)+(\nabla_wQ_h\phi, (I-\mathbf{Q}_s)\nabla_w\mathbf{v}_h) \Big)-\rho\Big((\nabla\mathbf{v}_0, (I-\mathbf{Q}_s)\beta\phi^t)\\
&-(\nabla_w\mathbf{v}_h, (I-\mathbf{Q}_s)\beta\phi^t) 
+\langle\mathbf{v}_0-\mathbf{v}_b, \mathbf{Q}_s(\beta\phi^t)\cdot\mathbf{n}\rangle+((I-Q_0)\phi, (\beta\cdot\nabla_w)\mathbf{v}_h) \Big)\\
&-\langle(I-Q_h^p)\xi, (\mathbf{v}_0-\mathbf{v}_b)\cdot\mathbf{n}\rangle+s_1(Q_h\phi,\mathbf{v}_h)\\
=&(e_{h0}^u, \mathbf{v}_0)+L^3_{\phi,\xi}(\mathbf{v}_h),
\end{aligned}
\end{equation}
where we have used $\langle\beta\phi^t, \mathbf{v}_b\cdot\mathbf{n}\rangle=0$ and 
$\langle\beta\phi^t, \mathbf{v}_0\cdot\mathbf{n} \rangle=\langle\mathbf{v}_0, \beta\phi^t\cdot\mathbf{n} \rangle$.
Here $L^3_{\phi,\xi}(\mathbf{v}_h)$ is defined as follows: 
\begin{equation*}
\begin{aligned}
&L^3_{\phi,\xi}(\mathbf{v}_h)=\mu\Big(\langle(I-\mathbf{Q}_s)\nabla\phi\cdot\mathbf{n}, \mathbf{v}_0-\mathbf{v}_b\rangle-((I-\mathbf{Q}_s)\nabla\phi,\nabla\mathbf{v} _0)\\
&+((I-Q_0)\phi, \nabla\cdot \mathbf{Q}_s\nabla_w\mathbf{v}_h)+(\nabla_wQ_h\phi, (I-\mathbf{Q}_s)\nabla_w\mathbf{v}_h) \Big)-\rho\Big((\nabla\mathbf{v}_0, (I-\mathbf{Q}_s)\beta\phi^t)\\
&-(\nabla_w\mathbf{v}_h, (I-\mathbf{Q}_s)\beta\phi^t) 
+\langle\mathbf{v}_0-\mathbf{v}_b, (\mathbf{Q}_s-I)(\beta\phi^t)\cdot\mathbf{n}\rangle+((I-Q_0)\phi, (\beta\cdot\nabla_w)\mathbf{v}_h) \Big)\\
&-\langle(I-Q_h^p)\xi, (\mathbf{v}_0-\mathbf{v}_b)\cdot\mathbf{n}\rangle+s_1(Q_h\phi,\mathbf{v}_h).
\end{aligned}
\end{equation*}
Taking $\mathbf{v}_h=e_h^u$ in \eqref{eq:project_term_st_dual} and using the error equation \eqref{eq:error_eq_st} lead to 
\begin{equation}\label{eq:l2_eq_ve_st}
\begin{aligned}
&\Vert e_{h0}^u\Vert ^2=\mu(\nabla_wQ_h\phi, \nabla_we_h^u)+\rho(Q_0\phi,(\beta\cdot\nabla_w)e_h^u)
-(Q_h^p \xi, \nabla_w\cdot e_{h}^u)\\
&+s_1(Q_h\phi, e_{h}^u)-L^3_{\phi,\xi}(e_h^u)\\
=&L^1_{u,p}(Q_h\phi)-L^2_{u,p}(Q_h\xi)+s_2(e_h^p, Q_h\xi)+(e_h^p, \nabla_w\cdot Q_h\phi)-L^3_{\phi,\xi}(e_h^u)\\
=&L^1_{u,p}(Q_h\phi)-L^2_{u,p}(Q_h\xi)+s_2(e_h^p, Q_h\xi)+\langle (I-Q_b)\phi, e_h^p\cdot\mathbf{n}\rangle-L^3_{\phi,\xi}(e_h^u),
\end{aligned}
\end{equation}
where we also used the assumption of $n\leq{k+1}$.
By employing $\Vert \nabla_wQ_h\phi \Vert_1\leq C\Vert \phi\Vert_2$, we deduce 
\begin{equation*}
\begin{aligned}
\left|L^1_{u,p}(Q_h\phi)\right|\leq& C(h^{s+1}\Vert\mathbf{u}\Vert_{s+2}+h^{k}\Vert\mathbf{u}\Vert_{k+1}+h^{n+1}\Vert p\Vert_{n+1})h\Vert\phi\Vert_2,\\
\left|L^2_{u,p}(Q_h\xi)\right|\leq&C(h^{j+1}\Vert\mathbf{u}\Vert_{j+1}+h^{n+1}h^{\alpha}\Vert p\Vert_{n+1})\Vert\xi\Vert_1,\\
\left|\langle (I-Q_b)\phi, e_h^p\cdot\mathbf{n}\rangle\right|\leq&Ch\Vert\phi\Vert_2\Vert e_h^p\Vert,\\
\left|s_2(e_h^p, Q_h\xi)\right|\leq&Ch\Vert e_h^p\Vert\Vert \xi\Vert_1.\\
\end{aligned}
\end{equation*}
Similar to the estimate in \eqref{eq:error_term_es_st}, we have 
\begin{equation*}
\begin{aligned}
\left|L^3_{\phi,\xi}(e_h^u)\right|\leq&Ch(\Vert\phi\Vert_2+\Vert \xi\Vert_1)\3bar e_h^u\3bar.\\
\end{aligned}
\end{equation*}
Thus, from the regularity of dual problem, we get
\begin{equation*}
\begin{aligned}
\Vert e_{h0}^u\Vert \leq&C\Big(h^{s+2}\Vert\mathbf{u}\Vert_{s+2}+h^{k+1}\Vert\mathbf{u}\Vert_{k+1}+h^{j+1}\Vert\mathbf{u}\Vert_{j+1}
+h^{n+1}(h+h^{\alpha})\Vert p\Vert_{n+1}\\
&+(h+h^{\alpha})\Vert e_h^p\Vert+h\3bar e_h^u\3bar \Big).\\
\end{aligned}
\end{equation*}
By combining the above inequality with Theorem \ref{th:ste_energy_er}, we obtain the desired error estimate for the velocity approximation in $L^2$ norm.
\end{proof}

\begin{remark}
When $n\leq j$, the parameter $\mu$ can be chosen as $\mu=0$; i.e. without including the stability term $s_2$ in the {\em{g}}WG scheme \eqref{sch:WG_FEM}. It should be pointed out that both the {\em{inf-sup}} condition and the convergence are established by assuming ${k-1}\leq n\leq \max\{m, k+1\}$.
In particular, the velocity estimate in $L^2$ norm holds true with $j\geq0$. A proof for this last claim can be given by following the above process; interested readers are referred to \cite{qi2021} for details.
\end{remark}

\section{{g}WG for Evolutionary Oseen Equation}\label{sec:4}
Now, we return to consider the evolutionary case. For simplicity, we introduce the following bilinear form 
\begin{equation*}
\begin{aligned}
a(\mathbf{v}_h, \chi_h)&=\mu(\nabla_w\mathbf{v}_h, \nabla_w\chi_h)
+\rho((\beta\cdot\nabla_w)\mathbf{v}_h,\chi_0)+s_1(\mathbf{v}_h, \chi_h).
\end{aligned}
\end{equation*}
We further introduce the projection $(E_h\mathbf{v}, E_h^p q)$ defined by the following equations:
\begin{equation}\label{stokes_pro}
\begin{aligned}
a(E_h\mathbf{v},\chi_h)-(E_h^p q,\nabla_w\cdot\chi_h)&=(-\mu\Delta\mathbf{v}+\rho(\beta\cdot\nabla)\mathbf{v}+\nabla q,\chi_0),~~\forall\chi_h\in\mathbf{V}_h^0,\\
(\nabla_w\cdot E_h\mathbf{v}, \psi)+s_2(E_h^p q,\psi)&=(\nabla\cdot\mathbf{v},\psi),~~~~~~~~~~~~~~~~~~~~~~~~~~~~~\forall\psi\in W_h.
\end{aligned}
\end{equation}
It is easy to see that \eqref{stokes_pro} is actually the generalized weak Galerkin scheme for the following problem: Find $(\mathbf{v},q)$ such that 
\begin{equation}\label{stokes_projection_problem}
\begin{aligned}
-\mu\nabla \cdot(\nabla \mathbf{v})+\rho(\beta\cdot\nabla)\mathbf{v}+\nabla q&=\mathbf{f}^{*},~~~~~~~~~~\mbox{in}~\Omega, \\
\nabla\cdot\mathbf{v}&=0,~~~~~~~~~~~\mbox{in}~\Omega,\\
\mathbf{v}&=\mathbf{g}, ~~~~~~~~~~~\mbox{on}~\partial\Omega.
\end{aligned}
\end{equation} 
From Theorems \ref{erres_steady_energy} and \ref{erres_steady_l2}, we have the following results. 
\begin{lemma}\label{lm:err_stokes_pro}
Let $(E_h\mathbf{v}, E_h^p q)$ be defined by \eqref{stokes_pro}.
Then there exists a constant such that
\begin{equation}\label{erres_stokes_pro}
\begin{aligned}
 \3bar Q_h\mathbf{v}-E_h\mathbf{v}\3bar^2&\leq CR_1,\\
 \Vert Q_h^p q-E_h^p q\Vert^2&\leq C(1+h^{\gamma+1}+h^{2(1-\alpha)})R_1,\\
\Vert Q_h\mathbf{v}-E_h\mathbf{v}\Vert ^2&\leq C(h^2+h^{2\alpha})(1+h^{\gamma+1}+h^{2(1-\alpha)})R_1.
\end{aligned}
\end{equation}
\end{lemma}

\subsection{Convergence of Semi-discrete Scheme}
To estimate the convergence of the evolutionary case \eqref{sch:evo_iWG_semi},
we denote the error terms of semi-discrete scheme by $e_h^u=Q_h\mathbf{u}-\mathbf{u}_h=\eta_h^u+\theta_h^u$ and $e_h^p=Q_h^p p-p_h=\eta_h^p+\theta_h^p$, 
where $\eta_h^u=Q_h\mathbf{u}-E_h\mathbf{u}$, $\theta_h^u=E_h\mathbf{u}-\mathbf{u}_h$ and 
$\eta_h^p=Q_h^p p-E_h^p p$, $\theta_h^p=E_h^p p-p_h$.

 From the initial condition in  {g}WG \eqref{sch:evo_iWG_semi}, we deduce $e_h^u(0)=0$.

\begin{theorem}\label{th:evo_er_semi}
Let $(\mathbf{u}_h, p_h)$ and $(\mathbf{u}, p)$ be the numerical solution of  {g}WG \eqref{sch:evo_iWG_semi} and exact solution of \eqref{evo_weakform}, respectively.
Assume that the velocity $\mathbf{u}, \mathbf{u}_t,\mathbf{u}_{tt}\in [H^{i}(\Omega)]^d$, $i=\max\{s+2,k+1,j+1\}$ and the pressure $p,p_t,p_{tt}\in H^{n+1}(\Omega)$. Let ${k-1}\leq n\leq \max\{m, k+1\}$ and $j\geq 1$.
Then there exists a constant such that
\begin{equation}\label{erres_evo_oseen_semi}
\begin{aligned}
\Vert e_h^{u}(t)\Vert ^2\leq& \Vert \theta_h^u(0)\Vert^2+C(h^2+h^{2\alpha})(1+h^{\gamma+1}+h^{2(1-\alpha)})\Big(R_1(t)+\int_0^tR_2(\bar{s})d\bar{s}\Big),\\
 \3bar e_h^{u}(t)\3bar^2\leq& \3bar \theta_h^u(0)\3bar^2+CR_1(t)\\
 &+C(h^2+h^{2\alpha})(1+h^{\gamma+1}+h^{2(1-\alpha)})\int_0^t\Big(R_2(\bar{s})+R_3(\bar{s})\Big)d\bar{s},\\
 \Vert e_h^{p}(t)\Vert^2\leq& C(1+h^{\gamma+1}+h^{2(1-\alpha)})\Big(R_1(t)+(1+h^{1+\gamma}+h^{1-\alpha})(h^2+h^{2\alpha})(R_2(t)\\
  &+\int_0^t(R_2(\bar{s})+R_3(\bar{s}))d\bar{s}+\mu\3bar\theta_{h}^u(0)\3bar^2+\frac{\rho}{2}\Vert\theta_{ht}^u(0)\Vert^2\Big).
\end{aligned}
\end{equation}
Here, 
\begin{equation*}
\begin{aligned}
R_1(t)=&h^{2(s+1)}(1+h^{-\gamma-1})\Vert \mathbf{u}(t)\Vert^2_{s+2}+h^{2k}(1+h^{-\gamma-1}+h^{\gamma+1})\Vert \mathbf{u}(t)\Vert^2_{k+1}\\
&+h^{2(n+1)}(1+h^{-\gamma-1}+h^{\alpha-1})\Vert p(t)\Vert^2_{n+1}+h^{2j}h^{1-\alpha}\Vert \mathbf{u}(t)\Vert^2_{j+1},
\end{aligned}
\end{equation*}
and
\begin{equation*}
\begin{aligned}
R_2(\bar{s})=&h^{2(s+1)}(1+h^{-\gamma-1})\Vert \mathbf{u}_t(\bar{s})\Vert^2_{s+2}+h^{2k}(1+h^{-\gamma-1}+h^{\gamma+1})\Vert \mathbf{u}_t(\bar{s})\Vert^2_{k+1}\\
&+h^{2(n+1)}(1+h^{-\gamma-1}+h^{\alpha-1})\Vert p_t(\bar{s})\Vert^2_{n+1}+h^{2j}h^{1-\alpha}\Vert \mathbf{u}_t(\bar{s})\Vert^2_{j+1},\\
R_3(\bar{s})=&h^{2(s+1)}(1+h^{-\gamma-1})\Vert \mathbf{u}_{tt}(\bar{s})\Vert^2_{s+2}+h^{2k}(1+h^{-\gamma-1}+h^{\gamma+1})\Vert \mathbf{u}_{tt}(\bar{s})\Vert^2_{k+1}\\
&+h^{2(n+1)}(1+h^{-\gamma-1}+h^{\alpha-1})\Vert p_{tt}(\bar{s})\Vert^2_{n+1}+h^{2j}h^{1-\alpha}\Vert \mathbf{u}_{tt}(\bar{s})\Vert^2_{j+1}.
\end{aligned}
\end{equation*}
Moreover, when $\gamma=-1$ and $\alpha=1$, optimal order of convergence is obtained.
\end{theorem}
\begin{proof}
Recall that we have obtained the estimate for $\eta_h^u$ and $\eta_h^p$ in Lemma \ref{lm:err_stokes_pro}. It remains to bound $\theta_h^u$ and $\theta_h^p$. To this end, note that for $\chi_h\in\mathbf{V}_h^0$ and $\psi\in W_h$, we have
\begin{equation}\label{eq:semi_er_term}
\begin{aligned}
\rho(\theta_{ht}^u,\chi_0)+a(\theta_h^u,\chi_h)-(\theta_h^p,\nabla_w\cdot\chi_h)&=-\rho(\eta_{ht}^u, \chi_0),\\
(\nabla_w\cdot\theta_h^u,\psi)+s_2(\theta_h^p,\psi)&=0.
\end{aligned}
\end{equation}
By choosing $\chi_h=\theta_h^u$ and $\psi=\theta_h^p$ in \eqref{eq:semi_er_term}, we have
\begin{equation*}
\begin{aligned}
&\frac{\rho}{2}\frac{d}{dt}\Vert\theta_{h}^u\Vert^2+\mu\Vert\nabla_w\theta_h^u\Vert^2+\rho((\beta\cdot\nabla_w)\theta_{h}^u, \theta_{h}^u)
+s_1(\theta_h^u,\theta_h^u)+s_2(\theta_h^p,\theta_h^p)\\
=&-\rho(\eta_{ht}^u, \theta_{h}^u),
\end{aligned}
\end{equation*}
Using the Cauchy-Schwarz inequality, Lemma \ref{lem:coercive}, and integrating over $(0,t)$, we arrive at
\begin{equation}\label{eq:semi_er_term2}
\begin{aligned}
&\frac{\rho}{2}\Vert\theta_{h}^u(t)\Vert^2+\int_0^t\Big(C_1\Vert\nabla_w\theta_h^u\Vert^2+{C_0\rho}\Vert\theta_{h}^u\Vert^2+\frac{1}{2}s_1(\theta_h^u,\theta_h^u)+s_2(\theta_h^p,\theta_h^p)\Big)d\bar{s}\\
\leq& \frac{\rho}{2}\Vert\theta_{h}^u(0)\Vert^2+C\int_0^t\Vert \eta_{ht}^u\Vert^2 d\bar{s}.
\end{aligned}
\end{equation}
The error estimate for velocity in $L^2$ norm is obtained at once from \eqref{eq:semi_er_term2} and  Theorem \ref{th:evo_er_semi},
 \begin{equation*}
\begin{aligned}
\Vert e_{h}^u(t)\Vert^2&\leq \Vert \eta_{h}^u(t)\Vert^2+\Vert \theta_{h}^u(t)\Vert^2\\
&\leq \Vert\theta_{h}^u(0)\Vert^2+ \Vert \eta_{h}^u(t)\Vert^2+C\int_0^t\Vert \eta_{ht}^u\Vert^2 d\bar{s}.
\end{aligned}
\end{equation*}

To estimate the velocity in energy norm, we have the following error equation by differentiating the equation \eqref{eq:semi_er_term}  
\begin{equation*}
\begin{aligned}
\rho(\theta_{htt}^u,\chi_0)+a(\theta_{ht}^u,\chi_h)-(\theta_{ht}^p,\nabla_w\cdot\chi_h)&=-\rho(\eta_{htt}^u, \chi_0),\\
(\nabla_w\cdot\theta_{ht}^u,\psi)+s_2(\theta_{ht}^p,\psi)&=0.
\end{aligned}
\end{equation*}
Taking $\chi_h=\theta_{ht}^u$ and $\psi=\theta_{ht}^p$ in the above equation yields
\begin{equation*}
\begin{aligned}
&\frac{\rho}{2}\frac{d}{dt}\Vert\theta_{ht}^u\Vert^2+\mu\Vert\nabla_w\theta_{ht}^u\Vert^2+\rho((\beta\cdot\nabla_w)\theta_{ht}^u, \theta_{ht}^u)+s_1(\theta_{ht}^u,\theta_{ht}^u)+s_2(\theta_{ht}^p,\theta_{ht}^p)\\
=&-\rho(\eta_{htt}^u, \theta_{ht}^u).
\end{aligned}
\end{equation*}
By integrating over $(0, t)$, we have
\begin{equation}\label{eq:semi_er_term3}
\begin{aligned}
&\frac{\rho}{2}\Vert\theta_{ht}^u(t)\Vert^2+\int_0^t\Big(C_1\Vert\nabla_w\theta_{ht}^u\Vert^2+{C_0\rho}\Vert\theta_{ht}^u\Vert^2+\frac{1}{2}s_1(\theta_{ht}^u,\theta_{ht}^u)+s_2(\theta_{ht}^p,\theta_{ht}^p)\Big)d\bar{s}\\
\leq& \frac{\rho}{2}\Vert\theta_{ht}^u(0)\Vert^2+C\int_0^t\Vert \eta_{htt}^u\Vert^2 d\bar{s}.
\end{aligned}
\end{equation}
Note the following equation
 \begin{equation*}
\begin{aligned}
&\frac{1}{2}\frac{d}{dt}\Big(\mu\Vert\nabla_w\theta_{h}^u\Vert^2+s_1(\theta_{h}^u,\theta_{h}^u)+s_2(\theta_{h}^p,\theta_{h}^p) \Big)\\
=&(\mu\nabla_w\theta_{ht}^u, \nabla_w\theta_{h}^u)+s_1(\theta_{ht}^u,\theta_{h}^u)+s_2(\theta_{ht}^p,\theta_{h}^p).
\end{aligned}
\end{equation*}
Then, it follows from \eqref{eq:semi_er_term2} and \eqref{eq:semi_er_term3} that
 \begin{equation}\label{eq:semi_er_term4}
\begin{aligned}
&\mu\Vert\nabla_w\theta_{h}^u(t)\Vert^2+s_1(\theta_{h}^u(t),\theta_{h}^u(t))+s_2(\theta_{h}^p(t),\theta_{h}^p(t))\\
\leq&\mu\Vert\nabla_w\theta_{h}^u(0)\Vert^2+s_1(\theta_{h}^u(0),\theta_{h}^u(0))+s_2(\theta_{h}^p(0),\theta_{h}^p(0))
+\int_0^t\Big(\mu\Vert\nabla_w\theta_{ht}^u\Vert^2\\
&+s_1(\theta_{ht}^u,\theta_{ht}^u)+s_2(\theta_{ht}^p,\theta_{ht}^p) \Big)d\bar{s}
+\int_0^t\Big(\mu\Vert\nabla_w\theta_{h}^u\Vert^2+s_1(\theta_{h}^u,\theta_{h}^u)+s_2(\theta_{h}^p,\theta_{h}^p) \Big)d\bar{s}\\
\leq&\mu\Vert\nabla_w\theta_{h}^u(0)\Vert^2+s_1(\theta_{h}^u(0),\theta_{h}^u(0))+s_2(\theta_{h}^p(0),\theta_{h}^p(0))\\
&+{\rho}\Vert\theta_{h}^u(0)\Vert^2+C\int_0^t\Vert \eta_{ht}^u\Vert^2 d\bar{s}+
{\rho}\Vert\theta_{ht}^u(0)\Vert^2+C\int_0^t\Vert \eta_{htt}^u\Vert^2 d\bar{s},
\end{aligned}
\end{equation}
which, together with \eqref{eq:semi_er_term4}, leads to the error estimate for the velocity in energy norm:
  \begin{equation*}
\begin{aligned}
\3bar e_{h}^u(t)\3bar^2\leq& \3bar \eta_{h}^u(t)\3bar^2+\3bar \theta_{h}^u(t)\3bar^2\\
\leq& \3bar \eta_{h}^u(t)\3bar^2+\3bar \theta_{h}^u(0)\3bar^2\\
&+{\rho}\Vert\theta_{h}^u(0)\Vert^2+C\int_0^t\Vert \eta_{ht}^u\Vert^2 d\bar{s}+
{\rho}\Vert\theta_{ht}^u(0)\Vert^2+C\int_0^t\Vert \eta_{htt}^u\Vert^2 d\bar{s}.
\end{aligned}
\end{equation*}
In the following, we shall estimate the pressure in $L^2$ by using the {\em{inf-sup}} condition in Lemma \ref{lem:inf-sup}  and the equations \eqref{eq:semi_er_term}:
 \begin{equation*}
\begin{aligned}
&\Vert \theta_h^p(t)\Vert^2\leq 2b(\mathbf{v}_h, \theta_h^p(t))+Ch^{1-\alpha}s_2(\theta_h^p(t),\theta_h^p(t))\\
=&2\rho(\theta_{ht}^u(t),\mathbf{v}_0)+2a(\theta_h^u(t),\mathbf{v}_h)-2\rho(\eta_{ht}^u(t), \mathbf{v}_0)+Ch^{1-\alpha}s_2(\theta_h^p(t),\theta_h^p(t))\\
\leq&C\Big( \Vert\theta_{ht}^u(t)\Vert+\Vert\eta_{ht}^u(t)\Vert+\3bar\theta_{h}^u(t)\3bar\Big)\3bar\mathbf{v}_h\3bar
+Ch^{1-\alpha}s_2(\theta_h^p(t),\theta_h^p(t))\\
\leq&C\Big( \Vert\theta_{ht}^u(t)\Vert+\Vert\eta_{ht}^u(t)\Vert+\3bar\theta_{h}^u(t)\3bar\Big)(1+h^{\frac{1-\gamma}{2}})\Vert \theta_h^p(t)\Vert
+Ch^{1-\alpha}s_2(\theta_h^p(t),\theta_h^p(t)).
\end{aligned}
\end{equation*}
Here we have used the fact $\Vert\mathbf{v}_0\Vert\leq C\3bar\mathbf{v}_h\3bar$.
Hence, from the Cauchy-Schwarz inequality and \eqref{eq:semi_er_term3}, \eqref{eq:semi_er_term2}, \eqref{eq:semi_er_term4} we have
 \begin{equation}\label{eq:semi_er_term5}
\begin{aligned}
&\Vert \theta_h^p(t)\Vert^2
\leq C(1+h^{1+\gamma})\Big( \Vert\theta_{ht}^u(t)\Vert^2+\Vert\eta_{ht}^u(t)\Vert^2+\3bar\theta_{h}^u(t)\3bar^2\Big)
+Ch^{1-\alpha}s_2(\theta_h^p(t),\theta_h^p(t))\\
\leq&C(1+h^{1+\gamma}+h^{1-\alpha})\Big( \int_0^t(\Vert\eta_{htt}^u\Vert^2+\Vert\eta_{ht}^u\Vert^2)d\bar{s}+\Vert\eta_{ht}^u(t)\Vert^2\\
&+\mu\3bar\theta_{h}^u(0)\3bar^2+\frac{\rho}{2}\Vert\theta_{ht}^u(0)\Vert^2\Big).
\end{aligned}
\end{equation}
Similar to the estimate for the velocity, we have 
 \begin{equation*}
\begin{aligned}
&\Vert e_h^p(t)\Vert^2\leq \Vert \eta_h^p(t)\Vert^2+\Vert \theta_h^p(t)\Vert^2,
\end{aligned}
\end{equation*}
which, together with the estimates in Theorem \ref{th:evo_er_semi}, completes the convergence proof for the semi-discrete scheme.
\end{proof}

\subsection{Convergence of Fully-discrete Scheme}
Now, we turn our attention to the fully-discrete scheme \eqref{sch:evo_iWG_fully}. The error term is denoted by
 $$e_u^{\bar{n}}=Q_h\mathbf{u}(t^{\bar{n}})-\mathbf{u}^{\bar{n}}=\eta_u^{\bar{n}}+\theta_u^{\bar{n}},$$
 where $\eta_u^{\bar{n}}=Q_h\mathbf{u}(t^{\bar{n}})-E_h\mathbf{u}(t^{\bar{n}})$ and $\theta_u^{\bar{n}}=E_h\mathbf{u}(t^{\bar{n}})-\mathbf{u}^{\bar{n}}$.
 Similarly, we write $e_p^{\bar{n}}=Q_h^p p(t^{\bar{n}})-p^{\bar{n}}=\eta_p^{\bar{n}}+\theta_p^{\bar{n}}$, and $\eta_p^{\bar{n}}=Q_h^p p(t^{\bar{n}})-E_h^p p(t^{\bar{n}})$, $\theta_p^{\bar{n}}=E_h^p p(t^{\bar{n}})-p^{\bar{n}}$.
 
\begin{theorem}\label{th:evo_er_fully}
Let $(\mathbf{u}^{{\bar{n}}+1}, p^{{\bar{n}}+1})$ and $(\mathbf{u}, p)$ be the numerical solution of  {g}WG \eqref{sch:evo_iWG_fully} and exact solution of \eqref{evo_weakform},
respectively.
Assume that the velocity $\mathbf{u}, \mathbf{u}_t,\mathbf{u}_{tt}\in [H^{i}(\Omega)]^d$, $i=\max\{s+2,k+1,j+1\}$ and the pressure $p,p_t,p_{tt}\in H^{n+1}(\Omega)$. Let ${k-1}\leq n\leq \max\{m, k+1\}$ and $j\geq 1$. 
Then there exists a constant such that
\begin{equation}\label{erres_evo_oseen_fully}
\begin{aligned}
\Vert e_{u}^{N+1}\Vert ^2&\leq \Vert \theta_u^0\Vert^2+C\tau^2\int_{t^0}^{t^{N+1}}\Vert \mathbf{u}_{tt}\Vert^2 d\bar{s}
+C(h^2+h^{2\alpha})(1+h^{\gamma+1}+h^{2(1-\alpha)})\\
&\Big(R_1(t^{N+1})+\int_{t^0}^{t^{N+1}}R_2(s)d\bar{s}\Big),\\
 \3bar e_u^{N+1} \3bar^2&\leq  \mu\Vert\nabla_w\theta_u^{0}\Vert^2+s_1(\theta_u^{0}, \theta_u^{0})+s_2(\theta_p^{0}, \theta_p^{0})+\rho\Vert \bar{\partial}_{t}\theta_u^{1}\Vert^2+{\rho}\Vert \theta_u^{0}\Vert^2+CR_1(t^{N+1})\\
&+C(h^2+h^{2\alpha})(1+h^{\gamma+1}+h^{2(1-\alpha)})\int_{t^{0}}^{t^{N+1}}R_2(\bar{s})+R_3(\bar{s})d\bar{s}\\
&+C\tau^2\int_{t^{0}}^{t^{N+1}}\Vert\mathbf{u}_{tt} \Vert^2+\Vert\mathbf{u}_{ttt} \Vert^2d\bar{s},
\end{aligned}
\end{equation}
and 
\begin{equation}\label{erres_evo_oseen_fully2}
\begin{aligned}
&\Vert e_{p}^{N+1}\Vert ^2\leq C(1+h^{1+\gamma}+h^{1-\alpha}+h^{2(1-\alpha)})\Big(\3bar\theta_u^{0}\3bar^2+s_2(\theta_p^{0}, \theta_p^{0})+\rho\Vert \bar{\partial}_{t}\theta_u^{1}\Vert^2+{\rho}\Vert \theta_u^{0}\Vert^2\\
&+R_1(t^{N+1})+\int_{t^{0}}^{t^{N+1}}R_2(\bar{s})+R_3(\bar{s})d\bar{s}
+\tau^2\int_{t^{0}}^{t^{N+1}}\Vert\mathbf{u}_{tt} \Vert^2+\Vert\mathbf{u}_{ttt} \Vert^2d\bar{s}\\
&+\frac{1}{\tau'}\int _{t^{N}}^{t^{N+1}}R_2(\bar{s})d\bar{s}+\tau'\int _{t^{N}}^{t^{N+1}}\Vert\mathbf{u}_{tt}\Vert^2d\bar{s}\Big),
\end{aligned}
\end{equation}
where $\tau'=t^{N+1}-t^N$.

Moreover, when $\gamma=-1$ and $\alpha=1$, optimal order of convergence is obtained.
\end{theorem}
\begin{proof}
Observe that the estimate for $\eta_u^{{\bar{n}}+1}$ and $\eta_p^{{\bar{n}}+1}$ has been established in Lemma \ref{lm:err_stokes_pro}. Thus, it suffices to handle $\theta_u^{{\bar{n}}+1}$ and $\theta_p^{{\bar{n}}+1}$. Note the following equations for all $\chi_h\in\mathbf{V}_h^0$ and $\psi\in W_h$
 \begin{equation}\label{eq:fully_er_term}
\begin{aligned}
&\rho(\bar{\partial}_t\theta_u^{{\bar{n}}+1}, \chi_0)+a(\theta_u^{{\bar{n}}+1}, \chi_h)-(\theta_p^{{\bar{n}}+1},\nabla_w\cdot\chi_h)\\
=&\rho(\bar{\partial}_t\eta_u^{{\bar{n}}+1}, \chi_0)+\rho(\bar{\partial}_t\mathbf{u}(t^{{\bar{n}}+1})-\mathbf{u}_t(t^{{\bar{n}}+1}), \chi_0),\\
&(\nabla_w\cdot\theta_u^{{\bar{n}}+1},\psi)+s_2(\theta_p^{{\bar{n}}+1}, \psi)=0.
\end{aligned}
\end{equation}
By choosing $\chi_h=\theta_u^{{\bar{n}}+1}$ and $\psi=\theta_p^{{\bar{n}}+1}$ in \eqref{eq:fully_er_term} and using Lemma \ref{lem:coercive}, we obtain
 \begin{equation*}
\begin{aligned}
&\frac{\rho}{2\tau}(\Vert \theta_u^{{\bar{n}}+1}\Vert^2-\Vert \theta_u^{\bar{n}}\Vert^2)+C_1\Vert\nabla_w\theta_u^{{\bar{n}}+1}\Vert^2+{C_0\rho}\Vert \theta_u^{{\bar{n}}+1}\Vert^2\\
&+\frac{1}{2}s_1(\theta_u^{{\bar{n}}+1}, \theta_u^{{\bar{n}}+1})+s_2(\theta_p^{{\bar{n}}+1}, \theta_p^{{\bar{n}}+1})\\
\leq&C\Vert\bar{\partial}_t\eta_u^{{\bar{n}}+1}\Vert^2+C\Vert\bar{\partial}_t\mathbf{u}(t^{{\bar{n}}+1})-\mathbf{u}_t(t^{{\bar{n}}+1})\Vert^2\\
\leq&C\Vert\frac{1}{\tau}\int_{t^{\bar{n}}}^{t^{{\bar{n}}+1}}\eta_{ht}^ud\bar{s}\Vert^2+C\Vert\frac{1}{\tau}\int_{t^{\bar{n}}}^{t^{{\bar{n}}+1}}(t^{\bar{n}}-\bar{s})\mathbf{u}_{tt}d\bar{s} \Vert^2\\
\leq&C\frac{1}{\tau}\int_{t^{\bar{n}}}^{t^{{\bar{n}}+1}}\Vert\eta_{ht}^u\Vert^2d\bar{s}+C\tau\int_{t^{\bar{n}}}^{t^{{\bar{n}}+1}}\Vert\mathbf{u}_{tt} \Vert^2d\bar{s}.
\end{aligned}
\end{equation*}
Thus, summing the above inequality from $0$ to $N$ yields 
 \begin{equation}\label{eq:fully_er_term2}
\begin{aligned}
&{\rho}\Vert \theta_u^{N+1}\Vert^2+\tau\sum_{{\bar{n}}=0}^{N}\Big(2C_1\Vert\nabla_w\theta_u^{{\bar{n}}+1}\Vert^2+{2C_0\rho}\Vert \theta_u^{{\bar{n}}+1}\Vert^2
+s_1(\theta_u^{{\bar{n}}+1}, \theta_u^{{\bar{n}}+1})\\
&+2s_2(\theta_p^{{\bar{n}}+1}, \theta_p^{{\bar{n}}+1})\Big)\\
\leq&{\rho}\Vert \theta_u^{0}\Vert^2+C\int_{t^0}^{t^{N+1}}\Vert\eta_{ht}^u\Vert^2d\bar{s}+C\tau^2\int_{t^0}^{t^{N+1}}\Vert\mathbf{u}_{tt} \Vert^2d\bar{s}.
\end{aligned}
\end{equation}
Thus,
 \begin{equation*}
\begin{aligned}
\Vert e_u^{N+1}\Vert^2&\leq \Vert \eta_u^{N+1}\Vert^2+\Vert \theta_u^{N+1}\Vert^2\\
&\leq \Vert \eta_u^{N+1}\Vert^2+\Vert \theta_u^{0}\Vert^2+C\int_{t^0}^{t^{N+1}}\Vert\eta_{ht}^u\Vert^2d\bar{s}
+C\tau^2\int_{t^0}^{t^{N+1}}\Vert\mathbf{u}_{tt} \Vert^2d\bar{s}.
\end{aligned}
\end{equation*}
Combining \eqref{eq:fully_er_term2} with Lemma \ref{lm:err_stokes_pro} yields the estimate for the velocity in $L^2$ norm.

In order to estimate the velocity in energy norm, we rewrite \eqref{eq:fully_er_term} as follows
 \begin{equation*}
\begin{aligned}
&\rho(\bar{\partial}_{tt}\theta_u^{{\bar{n}}+1}, \chi_0)+a(\bar{\partial}_{t}\theta_u^{{\bar{n}}+1}, \chi_h)-(\bar{\partial}_{t}\theta_p^{{\bar{n}}+1},\nabla_w\cdot\chi_h)\\
=&\rho(\bar{\partial}_{tt}\eta_u^{{\bar{n}}+1}, \chi_0)+\rho(\bar{\partial}_{tt}\mathbf{u}(t^{{\bar{n}}+1})-\bar{\partial}_{t}\mathbf{u}_t(t^{{\bar{n}}+1}), \chi_0),\\
&(\nabla_w\cdot\bar{\partial}_{t}\theta_u^{{\bar{n}}+1},\psi)+s_2(\bar{\partial}_{t}\theta_p^{{\bar{n}}+1}, \psi)=0,
\end{aligned}
\end{equation*}
where $\bar{\partial}_{tt}\theta_u^{{\bar{n}}+1}=\frac{\bar{\partial}_{t}\theta_u^{{\bar{n}}+1}-\bar{\partial}_{t}\theta_u^{{\bar{n}}}}{\tau}$.
Taking $\chi_h=\bar{\partial}_{t}\theta_u^{{\bar{n}}+1}$ and $\psi=\bar{\partial}_{t}\theta_p^{{\bar{n}}+1}$ in the above identity leads to
 \begin{equation*}
\begin{aligned}
&\frac{\rho}{2\tau}(\Vert \bar{\partial}_{t}\theta_u^{{\bar{n}}+1}\Vert^2-\Vert \bar{\partial}_{t}\theta_u^{{\bar{n}}}\Vert^2)+{C_1}\Vert\nabla_w\bar{\partial}_{t}\theta_u^{{\bar{n}}+1}\Vert^2+{C_0\rho}\Vert \bar{\partial}_{t}\theta_u^{{\bar{n}}+1}\Vert^2\\
&+\frac{1}{2}s_1(\bar{\partial}_{t}\theta_u^{{\bar{n}}+1}, \bar{\partial}_{t}\theta_u^{{\bar{n}}+1})+s_2(\bar{\partial}_{t}\theta_p^{{\bar{n}}+1}, \bar{\partial}_{t}\theta_p^{{\bar{n}}+1})\\
\leq&C\Vert\bar{\partial}_{tt}\eta_u^{{\bar{n}}+1}\Vert^2+C\Vert\bar{\partial}_{tt}\mathbf{u}(t^{{\bar{n}}+1})-\bar{\partial}_{t}\mathbf{u}_t(t^{{\bar{n}}+1})\Vert^2\\
\leq&C\frac{1}{\tau}\int_{t^{{\bar{n}}-1}}^{t^{{\bar{n}}+1}}\Vert\eta_{htt}^u\Vert^2d\bar{s}+C\tau\int_{t^{{\bar{n}}-1}}^{t^{{\bar{n}}+1}}\Vert\mathbf{u}_{ttt} \Vert^2d\bar{s}.
\end{aligned}
\end{equation*}
Here, we have used 
$$\Vert\bar{\partial}_{tt}\eta_u^{{\bar{n}}+1}\Vert^2\leq\frac{1}{\tau}\int_{t^{{\bar{n}}-1}}^{t^{{\bar{n}}+1}}\Vert\eta^u_{htt}\Vert^2d\bar{s}$$
 and
$$\Vert\bar{\partial}_{tt}\mathbf{u}(t^{{\bar{n}}+1})-\bar{\partial}_{t}\mathbf{u}_t(t^{{\bar{n}}+1})\Vert^2\leq\tau\int_{t^{{\bar{n}}-1}}^{t^{{\bar{n}}+1}}\Vert\mathbf{u}_{ttt} \Vert^2d\bar{s}.$$
Summing ${\bar{n}}$ from $0$ to $N$ gives rise to
 \begin{equation}\label{eq:fully_er_term3}
\begin{aligned}
&{\rho}(\Vert \bar{\partial}_{t}\theta_u^{N+1}\Vert^2-\Vert \bar{\partial}_{t}\theta_u^{1}\Vert^2)+{\tau}\sum_{{\bar{n}}=0}^N\Big(2{C_1}\Vert\nabla_w\bar{\partial}_{t}\theta_u^{{\bar{n}}+1}\Vert^2+{2C_0\rho}\Vert \bar{\partial}_{t}\theta_u^{{\bar{n}}+1}\Vert^2\\
&+s_1(\bar{\partial}_{t}\theta_u^{{\bar{n}}+1}, \bar{\partial}_{t}\theta_u^{{\bar{n}}+1})+2s_2(\bar{\partial}_{t}\theta_p^{{\bar{n}}+1}, \bar{\partial}_{t}\theta_p^{{\bar{n}}+1})\Big)\\
\leq&C\int_{t^{0}}^{t^{N+1}}\Vert\eta_{htt}^u\Vert^2ds+C\tau^2\int_{t^{0}}^{t^{N+1}}\Vert\mathbf{u}_{ttt} \Vert^2ds.
\end{aligned}
\end{equation}
On the other side, we have by using the Cauchy-Schwarz inequality 
 \begin{equation*}
\begin{aligned}
&\frac{1}{2\tau}\Big(\mu\Vert\nabla_w\theta_u^{{\bar{n}}+1}\Vert^2-\mu\Vert\nabla_w\theta_u^{{\bar{n}}}\Vert^2+s_1(\theta_u^{{\bar{n}}+1}, \theta_u^{{\bar{n}}+1})-s_1(\theta_u^{{\bar{n}}}, \theta_u^{{\bar{n}}}) \\
&+s_2(\theta_p^{{\bar{n}}+1}, \theta_p^{{\bar{n}}+1})-s_2(\theta_p^{{\bar{n}}}, \theta_p^{{\bar{n}}})\Big)\\
\leq&\mu(\nabla_w\bar{\partial}_{t}\theta_u^{{\bar{n}}+1}, \nabla_w\theta_u^{{\bar{n}}+1})+s_1(\bar{\partial}_{t}\theta_u^{{\bar{n}}+1}, \theta_u^{{\bar{n}}+1})+s_2(\bar{\partial}_{t}\theta_p^{{\bar{n}}+1}, \theta_p^{{\bar{n}}+1})\\
\leq&\frac{1}{2}\Big(\mu\Vert\nabla_w\theta_u^{{\bar{n}}+1}\Vert^2+s_1(\theta_u^{{\bar{n}}+1}, \theta_u^{{\bar{n}}+1})+s_2(\theta_p^{{\bar{n}}+1}, \theta_p^{{\bar{n}}+1})\Big)\\
&+\frac{1}{2}\Big(\mu\Vert\nabla_w\bar{\partial}_{t}\theta_u^{{\bar{n}}+1}\Vert^2+s_1(\bar{\partial}_{t}\theta_u^{{\bar{n}}+1}, \bar{\partial}_{t}\theta_u^{{\bar{n}}+1})+s_2(\bar{\partial}_{t}\theta_p^{{\bar{n}}+1}, \bar{\partial}_{t}\theta_p^{{\bar{n}}+1})\Big).
\end{aligned}
\end{equation*}
Summing the above inequality from $0$ to $N$ leads to 
 \begin{equation*}
\begin{aligned}
&\mu\Vert\nabla_w\theta_u^{N+1}\Vert^2-\mu\Vert\nabla_w\theta_u^{0}\Vert^2+s_1(\theta_u^{N+1}, \theta_u^{N+1})-s_1(\theta_u^{0}, \theta_u^{0}) \\
&+s_2(\theta_p^{N+1}, \theta_p^{N+1})-s_2(\theta_p^{0}, \theta_p^{0})\\
\leq&\tau\sum_{{\bar{n}}=0}^N\Big(\mu\Vert\nabla_w\theta_u^{{\bar{n}}+1}\Vert^2+s_1(\theta_u^{{\bar{n}}+1}, \theta_u^{{\bar{n}}+1})+s_2(\theta_p^{{\bar{n}}+1}, \theta_p^{{\bar{n}}+1})\Big)\\
&+\tau\sum_{{\bar{n}}=0}^N\Big(\mu\Vert\nabla_w\bar{\partial}_{t}\theta_u^{{\bar{n}}+1}\Vert^2+s_1(\bar{\partial}_{t}\theta_u^{{\bar{n}}+1}, \bar{\partial}_{t}\theta_u^{{\bar{n}}+1})+s_2(\bar{\partial}_{t}\theta_p^{{\bar{n}}+1}, \bar{\partial}_{t}\theta_p^{{\bar{n}}+1})\Big).
\end{aligned}
\end{equation*}
Together with  \eqref{eq:fully_er_term2} and \eqref{eq:fully_er_term3}, we obtain
 \begin{equation}\label{eq:fully_er_term4}
\begin{aligned}
&\mu\Vert\nabla_w\theta_u^{N+1}\Vert^2+s_1(\theta_u^{N+1}, \theta_u^{N+1}) +s_2(\theta_p^{N+1}, \theta_p^{N+1})\\
\leq&\mu\Vert\nabla_w\theta_u^{0}\Vert^2+s_1(\theta_u^{0}, \theta_u^{0})+s_2(\theta_p^{0}, \theta_p^{0})+\rho\Vert \bar{\partial}_{t}\theta_u^{1}\Vert^2+{\rho}\Vert \theta_u^{0}\Vert^2\\
&+C\int_{t^{0}}^{t^{N+1}}\Vert\eta_{ht}^u\Vert^2+\Vert\eta_{htt}^u\Vert^2d\bar{s}
+C\tau^2\int_{t^{0}}^{t^{N+1}}\Vert\mathbf{u}_{tt} \Vert^2+\Vert\mathbf{u}_{ttt} \Vert^2d\bar{s}.
\end{aligned}
\end{equation}
Note that $\3bar e_u^{N+1}\3bar^2\leq\3bar \eta_u^{N+1}\3bar^2+\3bar \theta_u^{N+1}\3bar^2$. Thus, by 
combining the above inequality and Lemma \ref{lm:err_stokes_pro} we obtain the desired estimate for the velocity in energy norm.

To estimate the pressure in $L^2$ norm , we first use Lemma \ref{lem:inf-sup} and \eqref{eq:fully_er_term} to obtain
 \begin{equation*}
\begin{aligned}
&\Vert \theta_p^{N+1}\Vert^2\leq 2b(\mathbf{v}_h, \theta_p^{N+1})+Ch^{1-\alpha}s_2(\theta_p^{N+1},\theta_p^{N+1})\\
=&\rho(\bar{\partial}_t\theta_u^{N+1}, \mathbf{v}_0)+a(\theta_u^{N+1}, \mathbf{v}_h)-\rho(\bar{\partial}_t\eta_u^{N+1}, \mathbf{v}_0)-\rho(\bar{\partial}_t\mathbf{u}(t^{N+1})-\mathbf{u}_t(t^{N+1}),\mathbf{v}_0)\\
&+Ch^{1-\alpha}s_2(\theta_p^{N+1},\theta_p^{N+1}).
\end{aligned}
\end{equation*}
Next from the Cauchy-Schwarz inequality and \eqref{eq:fully_er_term2}, \eqref{eq:fully_er_term3}, \eqref{eq:fully_er_term4}, we have
 \begin{equation*}
\begin{aligned}
\Vert \theta_p^{N+1}\Vert^2\leq& C(1+h^{1+\gamma})\Big(\Vert\bar{\partial}_t\theta_u^{N+1}\Vert^2+\3bar\theta_u^{N+1}\3bar^2+\Vert\theta_u^{N+1}\Vert^2\\
&+\frac{1}{\tau'}\int _{t^{N}}^{t^{N+1}}\Vert\eta_{ht}^u\Vert^2d\bar{s}+\tau'\int _{t^{N}}^{t^{N+1}}\Vert\mathbf{u}_{tt}\Vert^2d\bar{s}\Big)
+Ch^{1-\alpha}s_2(\theta_p^{N+1},\theta_p^{N+1})\\
\leq& C(1+h^{1+\gamma}+h^{1-\alpha})\Big( \mu\Vert\nabla_w\theta_u^{0}\Vert^2+s_1(\theta_u^{0}, \theta_u^{0})+s_2(\theta_p^{0}, \theta_p^{0})+\rho\Vert \bar{\partial}_{t}\theta_u^{1}\Vert^2\\
&+{\rho}\Vert \theta_u^{0}\Vert^2+\int_{t^{0}}^{t^{N+1}}\Vert\eta_{ht}^u\Vert^2+\Vert\eta_{htt}^u\Vert^2d\bar{s}
+\tau^2\int_{t^{0}}^{t^{N+1}}\Vert\mathbf{u}_{tt} \Vert^2+\Vert\mathbf{u}_{ttt} \Vert^2d\bar{s}\\
&+\frac{1}{\tau'}\int _{t^{N}}^{t^{N+1}}\Vert\eta_{ht}^u\Vert^2d\bar{s}+\tau'\int _{t^{N}}^{t^{N+1}}\Vert\mathbf{u}_{tt}\Vert^2ds\Big),
\end{aligned}
\end{equation*}
where $\tau'=t^{N+1}-t^{N}$. This completes the desired estimate for the pressure approximation by using Lemma \ref{lm:err_stokes_pro} and the inequality $\Vert e_{p}^{N+1}\Vert^2\leq\Vert \eta_{p}^{N+1}\Vert^2+\Vert \theta_{p}^{N+1}\Vert^2$.
\end{proof}

\begin{remark}
For the case of $n\leq j$, the {\em{g}WG} scheme \eqref{sch:evo_iWG_fully} for the evolutionary equation can be formulated without including the stabilizer $s_2$; i.e. $\sigma=0$.  The corresponding error estimate can be obtained without any difficulty by following the above procedure.
\end{remark}

\section{Numerical Experiments}\label{sec:5}
In this section, we will present two computational examples to verify the convergence theory established in previous sections. One of the two examples is a steady-state Oseen equation and the other is an evolutionary one. 
We employ the elements $([P_1]^2,[P_0]^2,[P_1]^{2\times 2},P_0,P_0)$ and $([P_2]^2,[P_1]^2,[P_1]^{2\times2},P_1,P_1)$ in the numerical implementation. The parameters are set as follows: $\mu=1$, $\rho=1$ and $\zeta=1$. The numerical experiments are based on a regular family $\mathcal{T}_h$ of triangulations for the domain.

\begin{example}\label{ex:01}
Consider the steady-state Oseen equation in domain $\Omega=(0,1)^2$. The exact solution is given by $\mathbf{u}=[x^2y;-x y^2]$ and $p=(2x-1)(2y-1)$. The convective vector is set as $\beta=[-x+\sin{x}\sin{y};\cos{x} \cos{y}]$ so that $\nabla\cdot\beta=-1$. 
\end{example}

The results are shown in Tables \ref{TT01} and \ref{TT02}. For the $([P_1]^2,[P_0]^2,[P_1]^{2\times2},P_0,P_0)$  element, we observe a convergence of order $O(h)$ for the velocity in energy norm 
and the pressure in $L^2$ norm. The convergence is shown to be of order $O(h^2)$ for the velocity in $L^2$ norm in Table \ref{TT01}. For the $([P_2]^2,[P_1]^2,[P_1]^{2\times2},P_1,P_1)$ element, we obtained the expected order of convergence for both the velocity and the pressure in various norms, as demonstrated in Table \ref{TT02}. The numerical results are in great consistency with the convergence theory developed in previous sections.


\begin{table}[!th]
\renewcommand{\captionfont}{\footnotesize}
  \centering
  \small
  \begin{tabular}{ccccccc}
  \hline\noalign{\smallskip}
  \multicolumn{1}{c}{h}&\multicolumn{2}{c}{$\3bare_h^\mathbf{u}\3bar$}
    &\multicolumn{2}{c}{$||e^\mathbf{u}_h||$}
      &\multicolumn{2}{c}{$||e_h^p||$}\\
\cline{2-3}\cline{4-5}\cline{6-7}\noalign{\smallskip}
 &error  &order &error  &order &error  &order\\
  \hline
 1/8  &9.5305e-02  &~~~  &4.6205e-03   &~~~   &1.3175e-01  &~~~    \\
 1/16  &4.9781e-02 &0.93 &1.3090e-03   &1.81   &6.4353e-02  &1.03   \\
 1/32  &2.5309e-02 &0.97 &3.4292e-04    &1.93  &3.0501e-02   &1.07 \\
 1/64  &1.2727e-02 &0.99 &8.7119e-05    &1.97 &1.4586e-02   &1.06  \\
\hline
  \end{tabular}
 \caption{Example \ref{ex:01}: Convergence with elements $([P_1]^2,[P_0]^2,[P_1]^{2\times2},P_0,P_0)$: $\gamma=1$ and $\sigma=0$, i.e.  without the stabilizer $s_2$ }
 \label{TT01}
\end{table}

\begin{table}[!th]
\renewcommand{\captionfont}{\footnotesize}
  \centering
  \small
  \begin{tabular}{ccccccc}
  \hline\noalign{\smallskip}
  \multicolumn{1}{c}{h}&\multicolumn{2}{c}{$\3bare_h^\mathbf{u}\3bar$}
    &\multicolumn{2}{c}{$||e^\mathbf{u}_h||$}
      &\multicolumn{2}{c}{$||e_h^p||$}\\
\cline{2-3}\cline{4-5}\cline{6-7}\noalign{\smallskip}
 &error  &order &error  &order &error  &order\\
  \hline
 1/8  &8.4499e-03  &~~~  &3.3478e-04     &~~~    &1.1235e-02    &~~~    \\
 1/16  &2.1246e-03 &1.99 &4.1883e-05      &2.99    &2.8109e-03    &1.99   \\
 1/32  &5.3256e-04 &1.99 &5.2382e-06    &2.99  &7.0305e-04      &1.99   \\
 1/64  &1.3331e-04 &1.99 &6.5497e-07    &2.99  &1.7581e-04      &1.99   \\
\hline
  \end{tabular}
 \caption{Example \ref{ex:01}: Convergence with elements $([P_2]^2,[P_1]^2,[P_1]^{2\times2},P_1,P_1)$:  $\gamma=1$ and $\sigma=0$, i.e.  without the stabilizer $s_2$ }
 \label{TT02}
\end{table}

\begin{example}\label{ex:02}
Consider the evolutionary Oseen equation with domain $(0,1)^2\times(0,\bar{T}]$. The exact solution is chosen as $\mathbf{u}=e^{-t}[x^2y;-x y^2]$ and $p=\sin{t}(2x-1)(2y-1)$. The convective vector is set as $\beta=[-x+\sin{x}\sin{y};\cos{x} \cos{y}]$ and $\bar{T}=1$.
\end{example}

For the evolutionary case, the convergence depends on the mesh size $h$ and the time step $\tau$.
Tables \ref{TT03} and \ref{TT04} illustrate the order of convergence in terms of the mesh size $h$. It can be seen that the expected convergence order of $O(h)$ is illustrated for the velocity in energy norm 
and the pressure in $L^2$ norm. For the velocity in $L^2$ norm, Table \ref{TT03} shows a clear convergence of  $O(h^2)$.  For the $([P_2]^2,[P_1]^2,[P_1]^{2\times2},P_1,P_1)$ element, Table \ref{TT04} shows a convergence order of $O(h^2)$ for the velocity in energy norm and the pressure in $L^2$ norm. The optimal order of convergence $O(h^3)$ was clearly seen in Table \ref{TT04} for the velocity in $L^2$ norm.

Table \ref{TT05} demonstrates the convergence of the numerical scheme in time discretization. As the backward Euler scheme was implemented in the fully-discrete scheme,  the table shows a convergence of order $O(\tau)$ in time with expectation. All the results are consistent with the theory developed in this paper.

\begin{table}[!th]
\renewcommand{\captionfont}{\footnotesize}
  \centering
  \small
  \begin{tabular}{ccccccc}
  \hline\noalign{\smallskip}
  \multicolumn{1}{c}{h}&\multicolumn{2}{c}{$\3bare_h^\mathbf{u}\3bar$}
    &\multicolumn{2}{c}{$||e^\mathbf{u}_h||$}
      &\multicolumn{2}{c}{$||e_h^p||$}\\
\cline{2-3}\cline{4-5}\cline{6-7}\noalign{\smallskip}
 &error  &order &error  &order &error  &order\\
  \hline
 1/8  &5.6155e-02 &~~~~ &2.9496e-03   &~~~~   &1.0792e-01  &~~~~   \\
 1/16  &3.0217e-02 &0.89 &8.7776e-04    &1.74  &5.3055e-02   &1.02\\
 1/32  &1.5530e-02 &0.96 &2.3377e-04    &1.90 &2.5267e-02    &1.07  \\
 1/64 &7.8384e-03  &0.98 &5.9719e-05   &1.96  &1.2118e-02    &1.06    \\
\hline
  \end{tabular}
 \caption{Example \ref{ex:02}: Convergence with elements $([P_1]^2,[P_0]^2,[P_1]^{2\times2},P_0,P_0)$: $\tau=h^2$, $\gamma=1$ and $\sigma=0$, i.e.  without the stabilizer $s_2$ }
 \label{TT03}
\end{table}

\begin{table}[!th]
\renewcommand{\captionfont}{\footnotesize}
  \centering
  \small
  \begin{tabular}{ccccccc}
  \hline\noalign{\smallskip}
  \multicolumn{1}{c}{h}&\multicolumn{2}{c}{$\3bare_h^\mathbf{u}\3bar$}
    &\multicolumn{2}{c}{$||e^\mathbf{u}_h||$}
      &\multicolumn{2}{c}{$||e_h^p||$}\\
\cline{2-3}\cline{4-5}\cline{6-7}\noalign{\smallskip}
 &error  &order &error  &order &error  &order\\
  \hline
 1/4  &2.6240e-02  &~~~  &2.0985e-03     &~~~    &3.7737e-02    &~~~    \\
 1/8  &6.6697e-03 &1.97 &2.6389e-04      &2.99    &9.4291e-03    &2.00   \\
 1/16  &1.6786e-03 &1.99 &3.3072e-05    &2.99  &2.3582e-03      &1.99   \\
 1/32  &4.2090e-04 &1.99 &4.1448e-06    &2.99  &5.8976e-04      &1.99   \\
\hline
  \end{tabular}
 \caption{Example \ref{ex:02}: Convergence with elements $([P_2]^2,[P_1]^2,[P_1]^{2\times2},P_1,P_1)$: $\tau=h^2$, $\gamma=1$ and $\sigma=0$, i.e.  without the stabilizer $s_2$ }
 \label{TT04}
\end{table}

\begin{table}[!th]
\renewcommand{\captionfont}{\footnotesize}
  \centering
  \small
  \begin{tabular}{ccccccc}
  \hline\noalign{\smallskip}
  \multicolumn{1}{c}{$\tau$}&\multicolumn{2}{c}{$\3bare_h^\mathbf{u}\3bar$}
    &\multicolumn{2}{c}{$||e^\mathbf{u}_h||$}
      &\multicolumn{2}{c}{$||e_h^p||$}\\
\cline{2-3}\cline{4-5}\cline{6-7}\noalign{\smallskip}
 &error  &order &error  &order &error  &order\\
  \hline
 1/4  &5.9875e-04  &~~~  &7.6289e-05   &~~~   &3.2232e-03  &~~~    \\
 1/8  &2.8739e-04 &1.05 &3.6495e-05   &1.06   &1.5441e-03  &1.06   \\
 1/16  &1.4248e-04 &1.01 &1.7854e-05    &1.03  &7.5657e-04   &1.02\\
 1/32  &7.4106e-05 &0.94 &8.8315e-06    &1.01 &3.7568e-04   &1.00  \\
\hline
  \end{tabular}
 \caption{Example \ref{ex:02}: Convergence with elements $([P_2]^2,[P_1]^2,[P_1]^{2\times2},P_1,P_1)$: $h=1/128$, $\gamma=1$ and $\sigma=0$, i.e.  without the stabilizer $s_2$ }
 \label{TT05}
\end{table}

\section{Conclusions}
In this paper, a {g}WG method was introduced and analyzed for the evolutionary Oseen equation by using generalized weak gradient for velocity. This new numerical scheme was designed without using the usual skew-symmetric formulation for the convective term. Error estimates of optimal order were established for the new scheme with arbitrary combination of polynomials. The backward Euler discretization was employed in the fully discrete scheme. Numerical experiments were conducted to validate the theory developed in this paper.

  
\end{document}